\documentclass[12pt]{amsart}
\usepackage{amsmath,amsfonts,amsthm,amscd,amssymb,stmaryrd}
\newtheorem{theorem}{Theorem}
\newcommand{\beq}{\begin{eqnarray}}
\newcommand{\eeq}{\end{eqnarray}}

\newcommand{\R}{\mathbb{R}}

\newcommand{\tr}{\mathrm{tr}}

\newcommand{\C}{\mathbb{C}}

\newcommand{\ob}{\mathcal{O}}

\newcommand{\tlaph}{\tilde{\Delta}_{H}}
\newcommand{\td}{\tilde{d}}

\newcommand{\tdel}{\tilde{\delta}}

\newcommand{\gS}{\underline{\tilde{g}}}

\newcommand{\ten}{\mathcal{T}}

\newtheorem{proposition}[theorem]{Proposition}
\newtheorem{lemma}[theorem]{Lemma}

\newtheorem{corollary}[theorem]{Corollary}
\newtheorem{definition}[theorem]{Definition}
\newtheorem{remark}[theorem]{Remark}

\numberwithin{equation}{section}
\numberwithin{theorem}{section}
\usepackage{amsmath}
\usepackage{amssymb}
\usepackage{latexsym}

\begin{document}
\bibliographystyle{amsalpha} 
\title[Obstruction-flat ALE metrics]{Obstruction-flat asymptotically \\[2pt]  locally Euclidean 
metrics}
\author{Antonio G. Ache}
\author{Jeff A. Viaclovsky}
\address{Department of Mathematics, University of Wisconsin, Madison, WI 53706}
\email{ache@math.wisc.edu, jeffv@math.wisc.edu}
\thanks{Research partially supported by NSF Grant DMS-0804042}
\date{June 06, 2011. Revised September 28, 2011.}
\begin{abstract} 
We show that any asymptotically locally Euclidean (ALE) metric 
which is obstruction-flat or 
extended obstruction-flat must be ALE of a certain optimal order.  
Moreover, our proof applies to very general elliptic systems 
and in any dimension $n \geq 3$. The proof is based on the technique 
of Cheeger-Tian for Ricci-flat metrics. 
We also apply this method to obtain a singularity
removal theorem for (extended) obstruction-flat  
metrics with isolated $C^0$-orbifold singular points. 
\end{abstract}
\maketitle
\setcounter{tocdepth}{1}
\tableofcontents

\section{Introduction}
We first recall the definition of an ALE metric. 

\begin{definition}
\label{ALEdef}
{\em
 A complete Riemannian manifold $(M,g)$ 
is called {\em{asymptotically locally 
Euclidean}} or {\em{ALE}} of order $\tau$ if it has finitely 
many ends, and for each end there exists a finite subgroup 
$\Gamma \subset SO(n)$ 
acting freely on $\mathbf{R}^n \setminus B(0,R)$ and a 
diffeomorphism 
$\Psi : M \setminus K \rightarrow ( \mathbf{R}^n \setminus B(0,R)) / \Gamma$ 
where $K$ is a subset of $M$ containing all other ends, 
and such that under this identification, 
\begin{align}
\label{eqgdfale1in}
(\Psi_* g)_{ij} &= \delta_{ij} + O( r^{-\tau}),\\
\label{eqgdfale2in}
\ \partial^{|k|} (\Psi_*g)_{ij} &= O(r^{-\tau - k }),
\end{align}
for any partial derivative of order $k$, as
$r \rightarrow \infty$, where $r$ is the distance to some fixed basepoint.  
We say that $(M,g)$ is ALE of order $0$ if we can find a coordinate system 
as above with $(\Psi_*g)_{ij} = \delta_{ij} + o(1)$, 
and $\partial^{|k|} (\Psi_*g)_{ij} = o(r^{- k })$ for any $k \geq 1$ 
as $r \rightarrow \infty$. }
\end{definition}
ALE spaces are ubiquitous in modern geometric analysis, and 
we do not attempt to give a complete list of references here. 
A crucial result in the Ricci-flat case was obtained 
by Cheeger-Tian: if  $(M^n, g)$ is Ricci-flat  
ALE of order $0$, there exists a change of coordinates 
at infinity so that $(M^n, g)$ is ALE of order $n$, where 
$n$ is the dimension \cite{ct}.  This generalized and improved earlier work of 
Bando-Kasue-Nakajima \cite{BKN}, who employed improved 
Kato inequalities together with a Moser iteration argument. 
The Cheeger-Tian method has the advantage of finding the 
{\em{optimal}} order of curvature decay, without 
relying on Kato inequalities. 

 Another interesting class of metrics is that of Bach-flat 
scalar-flat ALE metrics in dimension $4$, or more generally any 
metric satisfying a system of the form 
\begin{align}
\label{kcc}
\Delta Ric = Rm * Ric, 
\end{align}
where the right hand side is shorthand for a contraction of the full curvature 
tensor with the Ricci tensor. 
In the case of anti-self-dual scalar-flat metrics, 
or scalar-flat metrics with harmonic curvature, it was proved in 
\cite{TV} that such spaces are ALE of order $\tau$
for any $\tau < 2$, using the technique of Kato inequalities. 
Subsequently, this was generalized to Bach-flat metrics 
and metrics with harmonic curvature in dimension $4$ 
in \cite{s}, using the Cheeger-Tian technique. 
In this paper, we will simplify and generalize the Streets
argument so that it also works for higher order systems
and yields the optimal ALE order. 
A simplification from \cite{s}
is that we do not need to perform the entire radial separation of 
variables on symmetric tensors to obtain the optimal decay 
rate. Rather, we show this optimal decay can be obtained 
directly in Euclidean coordinates without running into very complicated 
formulas in radial coordinates, see Proposition \ref{proplin1}.

\subsection{The ambient obstruction tensor}\label{obs}

\noindent
Let $(M^n,g)$ be an $n$-dimensional Riemannian manifold, where $n > 2$. 
Recall that the curvature tensor admits the decomposition 
\begin{align}
Rm = W + A_g \varowedge g,
\end{align}
where $W$ is the Weyl tensor, $\varowedge$ is the Kulkarni-Nomizu product, 
$A_g$ is the {\em{Schouten tensor}} defined as 
\begin{align}
A_g = \frac{1}{n-2} \left(  Ric -  \frac{R}{2(n-1)} g \right), 
\end{align}
where $R$ denotes the scalar curvature. 
Define the $n$\emph{-dimensional Bach Tensor} by (see \cite{ChangFang, g2})
\begin{align}
B_{ij}=\Delta A_{ij}-\nabla^{k}\nabla_{i}A_{kj}+A^{kl}W_{ikjl},
\label{eqobs1}
\end{align}
where $\Delta$ denotes the rough Laplacian (our convention is to use the 
analyst's Laplacian). 

If the dimension $n$ is even, then the \emph{ambient obstruction tensor} introduced in \cite{fg85, fg},
and denoted by 
$\ob$ is a symmetric $(0,2)$-tensor that has the following properties:
\begin{enumerate}
\item $\ob(g)$ is trace-free.
\item If $n=4$, $\ob_{ij}(g)$ equals $B_{ij}(g)$ where $B_{ij}(g)$ is the 
Bach tensor of $g$.
\item If $g$ is conformal to an Einstein metric then $\ob(g)=0$.
\item  $\ob(g)$ has an expansion of the form
\begin{align}
\ob_{ij}=\Delta^{\frac{n}{2}-2}B_{ij}+l.o.t.\label{eqext2_1},
\end{align}
where $l.o.t.$ denotes quadratic and higher order terms in curvature 
involving fewer derivatives.
\item $\ob(g)$ is variational, in fact $\ob(g)$ is the gradient of the functional 
\begin{align*}
\mathcal{F}(g)=\int_{M}Q_{g}dV_{g},  
\end{align*}
\noindent
where $Q_{g}$ is the \emph{Q-curvature} of $g$. In particular, 
$\ob(g)$ is divergence-free.
\end{enumerate}

\subsection{Extended obstruction tensors}\label{ext}

\noindent
If $(M^n,g)$ is even-dimensional, there is also a family of symmetric $(0,2)$-tensors 
called \emph{extended obstruction tensors} introduced in \cite{g} and denoted 
by $\Omega^{(k)}(g)$ where $1\le k\le\frac{n}{2}-2$ which have the following 
properties:  
\begin{enumerate} 
\item $\Omega^{(k)}(g)$ is trace-free.
\item When the dimension $n$ is seen as a formal parameter, $\Omega^{(k)}(g)$ 
has a pole at $n=2(k+1)$, and its residue at $n=2(k+1)$ is a multiple of 
the obstruction tensor in that dimension, for example,
\begin{align*}
\Omega^{(1)}=\frac{1}{4-n}B_{ij},
\end{align*}
and when $n=4$, $B_{ij}$ equals the obstruction tensor.
\item If $(M,g)$ is locally conformally flat then $\Omega^{(k)}(g)=0$. 
\item $\Omega^{(k)}(g)$ has an expansion of the form
\begin{align}
\Omega^{(k)}_{ij}=\frac{1}{(4-n)(6-n)\ldots(2k-n)}\Delta^{k-1}B_{ij}
+l.o.t.,\label{eqext2}
\end{align}
where $l.o.t.$ denotes quadratic and higher order terms in curvature 
involving fewer derivatives.
\end{enumerate}
To simplify notation, we define 
$\Omega^{(k)}(g) = \mathcal{O}(g)$ for $k = \frac{n}{2} - 1$.

The main theorem in this paper gives the optimal decay rate 
for obstruction-flat or extended obstruction-flat scalar-flat ALE metrics:
\begin{theorem}
\label{maint} Let $(M^n,g)$ be even-dimensional,  
scalar-flat, and $\Omega^{(k)}$-flat for some $k$ with 
$1\le k \le\frac{n}{2}-1$. 
If $(M^n, g)$ is ALE of order zero, then there exists a change of 
coordinates at infinity so that $g$ is ALE of order~$n - 2k$. 
\end{theorem}

The method of Cheeger-Tian is to show that after a suitable change 
of coordinates, $g$ may be written as $g = g_0 + h$,
where $g_0$ is Euclidean, and $h$ is divergence-free. 
One then considers the linearization of the (extended) obstruction 
tensor at the flat metric. In the divergence-free gauge,
this becomes becomes a power of the Laplacian (the trace is 
controlled using the scalar-flat condition).
An analysis of the decay rates of solutions of the gauged linearized 
equation, together with an estimate on the nonlinear terms in the equation, 
then yields Theorem \ref{maint}. 

  The main technical complication is that the assumption of
ALE of order $0$ does not directly yield a divergence-free 
gauge. As in \cite{ct}, we obtain initially a modified divergence-free gauge 
$\delta_t h = 0$ (see Section \ref{df}). In this gauge, we must rule 
out certain solutions of the linearized equation 
which we call {\em{degenerate solutions}} (see Definition \ref{degdef}).
Once these degenerate solutions are ruled out, we are able
to find a change of coordinates so that $(M,g)$ is ALE of 
order $\beta > 0$. 
This step requires a technique of Leon Simon called the 
{\em{Three Annulus Lemma}}, 
which was also employed by Cheeger-Tian \cite{ls1, ct}. We
generalize this technique so that it applies to higher-order 
equations. For this step, we show that Turan's Lemma 
implies the necessary estimates, which we prove in the Appendix.  
Once this step is complete, it is relatively easy to find a 
divergence-free gauge using standard Fredholm Theory, and 
then to prove the optimal decay order. 
This work is carried out in Sections~\ref{lin}--\ref{osec}.

\begin{remark} 
{\em
In the case of the obstruction tensor, which is 
conformally invariant, one may obtain many examples 
through the following construction. 
Let $(M^{n},g)$ be an even-dimensional compact Einstein 
manifold with positive scalar curvature, and let $G_x$ denote the Green's 
functions of the conformal Laplacian at a point $x$. 
The metric $\hat{g} = G_x^{p} g$, where $p = \frac{4}{n-2}$, is asymptotically 
flat and scalar-flat \cite{LeeParker}. Since Einstein spaces are obstruction-flat
\cite[Theorem 2.1]{g2}, $\hat{g}$ is also obstruction-flat and 
asymptotically flat of order at least $2$. If $(M^{n},g)$ is instead locally 
conformally flat with positive scalar curvature, the same construction 
yields an $\Omega^{(1)}$-flat asymptotically flat space of order at least $n-2$.}
\end{remark}

Our method applies to much more general systems than just
the obstruction tensors, and works in any dimension $n \geq 3$.
Given two tensor fields $A, B$, the notation $A*B$ will mean a linear combination of 
contractions of $A\otimes B$ yielding a symmetric $2$-tensor. 
\begin{theorem}
\label{maint2} Let $k = 1$ if $n = 3$, or 
$1\le k \le\frac{n}{2}-1$ if $n \geq 4$.
Assume that $(M,g)$ is scalar-flat, 
ALE of order $0$, and satisfies
\begin{align}
\Delta^{k}_{g} Ric 
= \sum_{j=2}^{k+1}\sum_{\alpha_{1}+\ldots+\alpha_{j}=2(k+1)-2j}
\nabla_{g}^{\alpha_{1}}Rm*\ldots*\nabla_{g}^{\alpha_{j}}Rm.
\label{eqsrt11in}
\end{align}
Then $(M,g)$ is ALE of order $n-2k$. 
\end{theorem}
For $k =1$, this is simply 
\begin{align}
\label{kcc2}
\Delta Ric = Rm * Rm.
\end{align}
We emphasize that this is more general than \eqref{kcc}, since 
the right hand side is allowed to be quadratic in the full 
curvature tensor. This is satisfied in particular  
by scalar-flat K\"ahler metrics and metrics with 
harmonic curvature in any dimension, and also 
anti-self-dual metrics in dimension $4$. 
These special cases were previously considered in \cite{Chen} 
using improved Kato inequalites and a Moser iteration 
technique. We emphasize that our argument yields the optimal decay rate 
without requiring any improved Kato inequalities, and
therefore applies to the more general system \eqref{eqsrt11in}. 
The optimal decay for scalar-flat 
anti-self-dual ALE metrics was previously considered in \cite[Proposition 13]{CLW}.
The case of extremal K\"ahler ALE metrics was considered in \cite{CW}.
As mentioned above, the cases of Bach-flat metrics and 
metrics with harmonic curvature in dimension $4$ were considered in 
\cite{s}. 
However, we note that \eqref{kcc2} is more general than \eqref{kcc}. 

\begin{remark}{\em We do not need such a strong 
requirement on the decay of partial derivatives of arbitrarily high order
in \eqref{eqgdfale2in} in Definition \ref{ALEdef}, but have assumed this 
here in the introduction for simplicity of stating the result. 
We only need to assume this up to a finite number of partial 
derivatives, see Remark \ref{remale1}}. 
\end{remark}

\subsection{Singularity removal}

The methods used to prove the above results can also be applied to analyze 
isolated singularities. Similar results were proved in 
\cite{BKN, Chen, CW, CLW, s, Tian, TV2}.  
We next recall the definition of a $C^{0}$-orbifold point. 
\begin{definition}\label{c0def}{\em
Let $g$ be a metric defined on $B_{\rho}(0) \setminus \{ 0 \}$, where 
$B_{\rho}(0)$ is a metric ball in a flat cone. We say that the origin 
is a $C^{0}$-orbifold point if there exists a coordinate 
system around the origin such that
\begin{align}
g_{ij}&=\delta_{ij}+o(1),\label{eqorb1in}\\
\partial^{l}g_{ij}&=o(r^{-|l|}),\label{eqorb2in}
\end{align}
for any multi-index $l$ with $|l|\ge 1$ as $r \rightarrow 0$. 
We say that the origin is a smooth orbifold point, if 
after lifting to the universal cover of 
$B_{\rho}(0) \setminus \{ 0 \}$, the metric extends to 
a smooth metric over the origin, after diffeomorphism.}
\end{definition}

\begin{remark}{\em
As in the ALE case, we will not need to assume  \eqref{eqorb2in} for 
partial derivatives of arbitrarily high order, only 
up to a certain finite number of derivatives, see Remark \ref{ordrm2}}.
\end{remark}

 Applying the Cheeger-Tian technique directly to the singularity, 
we obtain the following. 
\begin{theorem}
\label{tr1}
 Let $B_{\rho}(0)$ be as above and even-dimensional, and let 
$g$ be (extended) obstruction-flat  
in $B_{\rho}(0)\backslash\{0\}$ with constant scalar curvature. 
If the origin is a $C^0$-orbifold point for $g$, 
then the metric extends to a smooth orbifold metric in $B_{\rho}(0)$.
\end{theorem}
As in the ALE case, this theorem also applies to much more general 
higher-order systems:
\begin{theorem} 
\label{tr2}
Let $k = 1$ if $n = 3$, or 
$1\le k \le\frac{n}{2}-1$ if $n \geq 4$.
Assume that $(B_{\rho}(0) \setminus \{0\},g)$ has 
constant scalar curvature and satisfies
\begin{align}
\Delta^{k}_{g} Ric 
= \sum_{j=2}^{k+1}\sum_{\alpha_{1}+\ldots+\alpha_{j}=2(k+1)-2j}
\nabla_{g}^{\alpha_{1}}Rm*\ldots*\nabla_{g}^{\alpha_{j}}Rm.
\label{eqsrt11inorb}
\end{align}
If the origin is a $C^0$-orbifold point for $g$, 
then the metric extends to a smooth orbifold metric in $B_{\rho}(0)$.
\end{theorem}
Theorem \ref{tr2} will be proved in Section \ref{srt}, the proof of 
which uses the same method as that of Theorem \ref{maint}, with a 
few minor modifications.  
\subsection{Acknowledgements}
The authors would like to thank  Michael Singer and Gang Tian 
for invaluable assistance from the inception of this 
project. We would also like to thank Robin Graham for crucial help with the structure of the obstruction tensors, Jeff Streets for numerous helpful 
discussions, and Sigurd Angenent and Mikhail Feldman for several helpful comments 
regarding elliptic theory. We are greatly indebted to Fedja Nazarov 
for explaining how Turan's Lemma could be used to prove Lemma~\ref{posneg}. 
Finally, we would like to thank the anonymous referee for 
numerous suggestions which greatly improved the 
exposition of the paper.  
\section{Linearized obstruction tensor}\label{lin}
We begin with some notation: in the following $\delta$ will denote 
the divergence operator, which can act either on a symmetric $2$-tensor $h$, 
or on a $1$-form $\omega$. In the former case
$\delta h = \nabla^i h_{ij}$ and in the latter case $\delta \omega = \nabla^i \omega_i$.
The $L^2$-adjoint of $\delta$ will be denoted by 
$\delta^*$, which is $-(1/2)\mathcal{L}$, where $\mathcal{L}$ 
is the Lie derivative operator, 
defined by $(\mathcal{L} \omega)_{ij} = \nabla_i \omega_j + 
\nabla_j \omega_i$.  The trace of a symmetric $2$-tensor $h$ will be denoted
by $tr(h)$. 

We now analyze the linearizations of (\ref{eqext2_1}) and  (\ref{eqext2}). 
Note that from the dependence of the lower order terms in both equations on the 
curvature tensor it follows that the linearized equations 
$\left(\Omega^{(k)}\right)^{\prime}_{g_{0}}(h)=0$ for $1\le k\le \frac{n}{2}-1$ 
at a flat metric $g_{0}$ are equivalent to 
$\Delta^{k-1}B^{\prime}_{g_{0}}(h)=0$. 
With this observation we have the following, which holds in 
any dimension $n \geq 3$: 
\begin{proposition}
At a flat metric $g_{0}$ we have
\begin{align}
\begin{split}
&\Delta^{k-1}\left(B^{\prime}_{g_{0}}(h)\right)
=\Delta^{k-1}\left(-\frac{1}{2(n-2)}\Delta^{2}h
-\frac{1}{2(n-1)(n-2)}\nabla^{2}\Delta\tr(h)\right.\\
&\left.-\frac{1}{2(n-1)}\nabla^{2}\delta\delta h 
-\frac{1}{(n-2)}\Delta\delta^{*}\delta h+\right.
\left.\frac{1}{2(n-1)(n-2)}\left(\Delta^{2}\tr(h)
-\Delta\delta\delta h\right)g_{0}\right).\label{eqlin1}
\end{split}
\end{align}
\end{proposition}

\begin{proof} The Bach tensor can be written as 
$B_{ij}=\Delta A_{ij}-\nabla^{k}\nabla_{i}A_{kj}+A^{kl}W_{ikjl}$, 
and at a flat metric $\left(A^{kl}W_{ikjl}\right)^{\prime}_{g_{0}}(h)=0$ for any 
$h\in S^{2}\left(T^{*}\R^{n}\right)$, so then
\begin{align*}
B^{\prime}_{ij}(h)=\Delta A^{\prime}_{ij}(h)
-\nabla^{k}\nabla_{i}A^{\prime}_{jk}(h).
\end{align*}

\noindent
By the Bianchi identity and using again that $g_{0}$ is flat we have
\begin{align}
\begin{split}
\label{eqlin0_1}
-\nabla^{k}\nabla_{i}(A_{g_{0}}^{\prime})_{jk}(h)
&=-\nabla_{i}\nabla^{k}(A_{g_{0}})^{\prime}_{kj}(h)
= -\frac{1}{2(n-1)}\nabla_i \nabla_j R_{g_{0}}^{\prime}(h)\\
&=\frac{1}{2(n-1)}\nabla_i \nabla_j \left(\Delta\tr(h)-\delta(\delta h)\right).
\end{split}
\end{align}

\noindent
On the other hand
\begin{align}
\begin{split}\label{eqlin0_2}
&\Delta A_{g_{0}}^{\prime}(h)=\frac{1}{(n-2)}\Delta\left(Ric_{g_{0}}^{\prime}(h)
-\frac{R_{g_{0}}^{\prime}}{2(n-1)}(h)g_{0}\right) \\
&=\frac{1}{(n-2)}\Delta\left(-\frac{1}{2}\Delta h
-\frac{1}{2}\nabla^{2}\tr(h)-\delta^{*}\delta h
+\frac{1}{2(n-1)}\left[\Delta\tr(h)
-\delta(\delta h)\right]g_{0}\right).
\end{split}
\end{align}
Adding (\ref{eqlin0_1}) and (\ref{eqlin0_2}) together and taking
$\Delta^{k-1}$ we obtain (\ref{eqlin1}).

\end{proof}
The linearized equations 
are \emph{not} strictly elliptic due to the 
diffeomorphism invariance of the (extended) obstruction-flat equations. 
Note, however, that if $h$ is divergence free and if the trace of 
$h$ is harmonic then from (\ref{eqlin1}) the linearized equations reduce to 
\begin{align}
\Delta^{k+1}h&=0,\label{eqlin2}\\
\Delta\tr(h)&=0,\label{eqlin3}\\
\delta h&=0,\label{eqlin4}
\end{align}
and (\ref{eqlin2}) is an elliptic equation on $h$. We point out that we will 
\emph{not} be able to prescribe that $h$ be divergence-free at first, 
so we will follow the approach in \cite{ct} and introduce a modified 
divergence operator in Section~\ref{sfm}. 
The proof of Theorem~\ref{maint} relies on the following crucial proposition.
\begin{proposition}\label{proplin1}Let $n > 2$, 
and $h\in S^{2}(T^{*}\R^{n})$ be a solution on $\R^{n}{\backslash} B_{\rho}(0)$ 
for some $\rho>0$ of the system 
\begin{align}
\Delta^{k+1} h&=0,\label{eqlin5}\\
\delta h&=0,\label{eqlin6}
\end{align}
for $1\le k\le \frac{n}{2}-1$ (or for $k=1$ when $n=3$) satisfying $h=O(|x|^{1-\epsilon})$ with 
$\epsilon>0$. Then, $h$ can be written as 
\begin{align}
h=h_{c}+O\left(|x|^{-n+2k}\right),\label{eqlin6_1}
\end{align}
where $h_{c}$ is  constant. The result also holds for $k=0$ 
if in addition we assume that $\tr(h)=0$. 
\end{proposition}
\begin{proof} 
Consider first the case $1\le k\le\frac{n}{2}-2$.
Since 
the components of $h$ satisfy the scalar equation $\Delta^{k+1} f = 0$, 
using a classical expansion, we may write a 
solution $h$ of (\ref{eqlin5})-(\ref{eqlin6}) as 
\begin{align}
\label{hpex}
h(x)=h_{c}+\sum_{l=0}^{\infty}h_{l}(x),
\end{align}
where $h_{c}$ is constant and each $h_{l}$ is a homogeneous solution of 
(\ref{eqlin5})-(\ref{eqlin6}) of degree $2(k+1)-n-l$. If $h_{l}$ is one of 
such solutions, we can write
\begin{align}
h_{l}(x)
=|x|^{2(k+1)-n-l} \sum_{ (s,j) \in S_{k,l}} h_{s,j}(x),
\label{eqlin6_2}
\end{align}
where $S_{k,l} = \{(s,j): 2(j+1) - s = 2(k+1) -l, s \geq 0, 0 \leq j \leq k\}$, 
and the components of $h_{s,j}(x)$ in (\ref{eqlin6_2}) are spherical harmonics 
of degree $s$. Note that if $(s,j) \in S_{k,l}$, then 
$\Delta^{j+1} ( |x|^{2(k+1) - n - l} h_{s,j}(x)) = 0$. 
In order to prove the claim for $1\le k\le \frac{n}{2}-2$,
it suffices to show that $h_{0}=h_{1}\equiv 0$. For $l=0$, we have
\begin{align*}
(h_{0})_{ij}(x)=|x|^{2(k+1)-n}c_{ij},
\end{align*}
where $c_{ij}$ is constant. From  (\ref{eqlin6}) we obtain 
\begin{align*}
\sum_{i=1}^{n}c_{ij}x_{i}=0,
\end{align*}
for each $j$ and this clearly implies that $h_{0}$ is identically zero. For 
$h_{1}$ we have 
\begin{align*}
\left(h_{1}\right)_{ij}(x)=|x|^{2(k+1)-n}u_{ij}\left(\frac{x}{|x|^{2}}\right),
\end{align*}
where $u_{ij}(x)$ is a homogeneous polynomial of degree 1 that we will write as
\begin{align}
u_{ij}(x)=\sum_{l=1}^{n}A_{ijl}x_{l}.\label{eqlinmat}
\end{align}
From (\ref{eqlin6}) we have for every $j$
\begin{align*}
0=\sum_{i=1}^{n}\partial_{i}\left(|x|^{2(k+1)-n}u_{ij}(x\slash |x|^{2})\right)
=&-\sum_{i=1}^{n}(n-2k)x_{i}|x|^{2(k-1)-n}u_{ij}(x)\\
&+|x|^{2k-n}\sum_{i=1}^{n}\partial_{i}u_{ij}(x),
\end{align*}
which becomes
\begin{align}
\left(n-2k\right)\sum_{i=1}^{n}\sum_{l=1}^{n}A_{ijl}x_{i}x_{l}
=|x|^{2}\sum_{i=1}^{n}A_{iji}.\label{eqlin9}
\end{align}
For fixed $j$ and every $x\in\R^{n}$, with $x\ne 0$. If in (\ref{eqlin9}) 
we let $x$ be the vector with coordinates $x_{i}=\delta_{ip}$ for fixed $p$, 
one obtains the identity
\begin{align}
A_{pjp}=\frac{1}{n-2k}\sum_{l=1}^{n}A_{ljl}~\mbox{for fixed}~j.\label{eqlin10}
\end{align}
An obvious consequence of (\ref{eqlin10}) is that $A_{pjp}$ is independent of 
$p$ for fixed $j$, in particular, for every $p$ and $j$
\begin{align}
A_{pjp}=\frac{n}{n-2k}A_{pjp},
\end{align}
and then $A_{pjp}=0$ since $k\ne 0$. For the components of the form $A_{ljm}$ 
with $l\ne m $, the coefficient of $x_{l}x_{m}$ in the left-hand side of 
(\ref{eqlin9}) is $\left(n-2k\right)\left(A_{ljm}+A_{mjl}\right)$ while in 
the right-hand side there are no off-diagonal terms, so we conclude that  
\begin{align}
A_{ljm}=-A_{mjl}~\mbox{for}~l\ne m,\label{eqlin11}
\end{align}
If $l,j,m$ are all different we obtain from the symmetry of $A_{ljm}$ in $l,j$ 
and from (\ref{eqlin11}) the identity
\begin{align*}
A_{ljm}=-A_{mjl}=-A_{jml}=A_{lmj}=A_{mlj}= -A_{jlm}=-A_{ljm},
\end{align*}
therefore, in this case $A_{ljm}=0$. For the components of the form $A_{llj}$ 
and $A_{jll}$ when $l\ne j$, it is easy to see that $A_{llj}=-A_{ljl}$  and 
$A_{jll}=A_{ljl}$ and as we saw above this implies that both components are 
zero, so we conclude that all polynomials $u_{ij}$ are identically zero. 

For the case $k=\frac{n}{2}-1$ there is only one difference with the argument 
above: $h_0(x)$ is logarithmic, i.e., a solution of the form 
$h_{ij}(x)=\log(|x|)c_{ij}$ with $c_{ij}$ constant, however the condition 
$\delta h=0$ implies that $\sum_{j=1}^{n}c_{ij}x_{j}=0$ for every $i$ and 
hence $c_{ij}=0$ for all $i,j$ so this solution in fact does \textit{not} 
occur. 

For the case $k=1$ and $n=3$ we write $h$ as 
\begin{align}
h(x)=\sum_{l=0}^{\infty}h_{l}(x),
\end{align}

\noindent
where $h_{l}(x)$ is a homogeneous solution of degree $-l$ of $\Delta^{2}h(x)=0$ on $\R^{n}\backslash\{0\}$. In this case, the solution $h_{0}(x)$ has the form $h_{0}(x)=h_{C}+h_{0,1}(x)$
where the components of $h_{C}$ are constant and the components of $h_{0,1}$ are spherical harmonics of degree 1. The solutions $h_{l}(x)$ with $l\ge 1$  have the form $h_{l}(x)=|x|^{-l}\left(h_{l,l-1}(x)+h_{l,l+1}(x)\right)$ where the components of $h_{l,l\pm 1}(x)$ are spherical harmonics of degree $l\pm 1$. We only have to prove that if $\delta h_{0,1}(x)=0$ on $\R^{n}\backslash\{0\}$ then $h_{0,1}(x)\equiv 0$. For that purpose write the components of $h_{0,1}(x)$ as $(h_{0,1})_{ij}(x)=u_{ij}\left(\frac{x}{|x|}\right)$ where $u_{ij}(x)$ are linear functions given by (\ref{eqlinmat}). The condition $\delta h_{0,1}(x)=0$ for all $x\ne 0$ becomes
\begin{align*}
\sum_{i=1}^{3}\sum_{l=1}^{3}A_{ijl}x_{i}x_{l}=|x|^{2}\sum_{i=1}^{3}A_{iji},
\end{align*}
and we can argue as in the case $1\le k\le\frac{n}{2}-1$ to conclude that $A_{ijl}=0$ for all $i,j,l=1,2,3$. 

 The above proof can be extended to the case of $k =0$ provided 
the trace vanishes. However, we omit the proof, and instead refer the 
reader to the proof given in \cite[page 538]{ct}, which is 
an alternative argument using Obata's Theorem. 
\end{proof} 

\section{Nonlinear terms in the obstruction-flat systems}\label{rem}

\noindent
In this section we derive an expression for the error terms in the 
linearization of the (extended) obstruction tensors, i.e., the difference  
\begin{align}
&\Omega^{(k)}(g_{0}+h)-\Omega^{(k)}(g_{0})
-\left(\Omega^{(k)}\right)^{\prime}_{g_{0}}(h),
\end{align}
where $g_{0}$ is a flat metric in $\R^{n}$. 
Given two tensor fields $A,B$ by $A*B$ we mean a linear combination of contractions of $A\otimes B$ using the metric $g_{0}$, and  
for a positive integer $j$, $A^{-j}* B$ means contractions of $j$ copies of the 
inverse of $A$ with $B$.  

\begin{proposition}\label{proprem1}
Let $g_{0}$ be a flat metric on $\R^{n}$ and let $h\in S^{2}(T^{*}\R^{n})$
be such that $g_{0}+h$ is another Riemannian metric on $\R^{n}$. 
For the (extended) obstruction tensors $\Omega^{(k)}$ with 
$1\le k\le\frac{n}{2}-1$,   we have
\begin{align}
\begin{split}\label{eqrem10}
\Omega^{(k)}(g_{0}+h)
&=\left(\Omega^{(k)}\right)^{\prime}_{g_{0}}(h)- (g_{0}+h)^{-1} * h *\nabla_{g_0}^{2(k+1)}h\\
&+\sum_{j=2}^{\mathcal{I}_{k}}(g_{0}+h)^{-j}*\left(\sum_{\alpha_{1}+\ldots+\alpha_{j}=2(k+1)}
\nabla_{g_{0}}^{\alpha_{1}}h*\ldots*\nabla_{g_{0}}^{\alpha_{j}}h\right),
\end{split}
\end{align}
for some integer $\mathcal{I}_{k} > 2(k+1)$. For the scalar curvature we have, 
\begin{align}
\begin{split}\label{eqrem11_1}
R(g_{0}+h)&=R^{\prime}_{g_{0}}(h)+(g_{0}+h)^{-1}*h*\nabla_{g_{0}}^{2}h
+(g_{0}+h)^{-2}*\left( h*\nabla_{g_{0}}^{2}h + \nabla_{g_{0}}h*\nabla_{g_{0}}h\right)\\
&+(g_{0}+h)^{-3}*\left(\nabla_{g_{0}}h *\nabla_{g_{0}}h*h\right).
\end{split}
\end{align}
\end{proposition}
\begin{proof} 
For any tensor $T$, we have
\begin{align}
\nabla_{g_{0}+h}T=\nabla_{g_{0}}T+(g_{0}+h)^{-1}*\nabla_{g_{0}} h*T.\label{eqrem0}
\end{align}
From this it follows that for the $(1,3)$ curvature tensor
\begin{align}
Rm(g_{0}+h)=(g_{0}+h)^{-1}*\nabla^{2}_{g_{0}}h
+(g_{0}+h)^{-2}*\nabla_{g_{0}} h*\nabla_{g_{0}} h.\label{eqrem3}
\end{align}
It follows that the $Ric(g_{0}+h)$ has an expansion similar to \eqref{eqrem3}. 
For the scalar curvature, on the other hand,  we have
\begin{align}
\begin{split}\label{eqrem3'}
R(g_{0}+h)&=\left(g_{0}+h\right)^{-1}*\left(Ric(g_{0}+h)\right)\\
&=\left(g_{0}+h\right)^{-1}*\left((g_{0}+h)^{-1}*\nabla^{2}_{g_{0}}h
+(g_{0}+h)^{-2}*\nabla_{g_{0}} h*\nabla_{g_{0}} h\right).
\end{split}
\end{align}
Using the identity
\begin{align}
(g_{0}+h)^{-1}-g_{0}^{-1}=-g_{0}^{-1}*h*(g_{0}+h)^{-1},\label{eqrem11}
\end{align}
we obtain  (we will omit writing $g_0^{-1}$ from now on)
\begin{align*}
R(g_{0}+h)&=(g_{0}+h)^{-1}*\nabla^{2}_{g_{0}}h+\left(g_{0}+h\right)^{-2}*\left(h*\nabla^{2}_{g_{0}}h+\nabla_{g_{0}}h*\nabla_{g_{0}}h\right)\\
&+(g_{0}+h)^{-3}*h*\nabla_{g_{0}}h*\nabla_{g_{0}}h,
\end{align*}
and another application of \eqref{eqrem11} to the leading term yields
\begin{align}
\begin{split}\label{sgh}
R(g_{0}+h)&= \nabla^2_{g_0} h  + (g_{0}+h)^{-1}* h *\nabla^{2}_{g_{0}}h
+\left(g_{0}+h\right)^{-2}*\left(h*\nabla^{2}_{g_{0}}h+\nabla_{g_{0}}h*\nabla_{g_{0}}h\right)\\
&+(g_{0}+h)^{-3}*h*\nabla_{g_{0}}h*\nabla_{g_{0}}h.\\
\end{split}
\end{align}
Since the only term in \eqref{sgh} that contributes to the linearization of $R$ at $g_{0}$ is $\nabla^{2}_{g_{0}}h$, equation \eqref{eqrem11_1} follows. In order to find a 
similar expansion for the Bach tensor $B(g_{0}+h)$ we note that for any metric $g$, 
$B(g)$ can be written schematically as
\begin{align}\label{bachg}
B(g)=\Delta_{g} Ric_g+\nabla^{2}_{g} R_g+\left(\Delta_{g}R_g\right)g+Rm_g*Rm_g,
\end{align}
We first consider the term 
$\Delta_{g}Ric_g$ at $g=g_{0}+h$ in~\eqref{bachg}.
Taking a covariant derivative, again assuming that $g_0$ is flat, we have 
\begin{align}
\begin{split}\label{eqrem6}
\nabla_{g_{0}+h}Ric(g_{0}+h)&=(g_{0}+h)^{-1}*\left(\nabla^{3}_{g_{0}}h\right)
+(g_{0}+h)^{-2}*\nabla_{g_{0}}^{2}h*\nabla_{g_{0}}h\\
&+(g_{0}+h)^{-3}*\nabla_{g_{0}}h*\nabla_{g_{0}}h*\nabla_{g_{0}}h.
\end{split}
\end{align}
Taking another covariant derivative and one contraction with 
respect to $g_{0}+h$, we obtain
\begin{align}
\begin{split}\label{eqrem7}
&\Delta_{g_{0}+h}Ric(g_{0}+h)
=(g_{0}+h)^{-2}*\left(\nabla^{4}_{g_{0}}h\right)\\
&+(g_{0}+h)^{-3}*\left(\nabla_{g_{0}}^{2}h*\nabla_{g_{0}}^{2}h
+\nabla^{3}_{g_{0}}h*\nabla_{g_{0}}h\right)\\
&+(g_{0}+h)^{-4}*\nabla^{2}_{g_{0}}h*\nabla_{g_{0}}h*\nabla_{g_{0}}h\\
&+(g_{0}+h)^{-5}
\left(\nabla_{g_{0}}h*\nabla_{g_{0}}h*\nabla_{g_{0}}h*\nabla_{g_{0}}h\right).
\end{split}
\end{align}
Using \eqref{eqrem11} twice (as we did above for $R$),  
we conclude that
\begin{align}
\begin{split}
\Delta_{g_{0}+h}Ric(g_{0}+h)&= \Delta_{g_{0}}Ric^{\prime}_{g_{0}}(h) 
+  (g_0 + h)^{-1} * h * \nabla^4_{g_0} h\\
&+\sum_{j=2}^{5}\sum_{\alpha_{1}+\ldots+\alpha_{j}=4}(g_{0}+h)^{-j}*\nabla_{g_{0}}^{\alpha_{1}}h*\ldots*\nabla_{g_{0}}^{\alpha_{j}}h.
\end{split}
\end{align}
From similar computations for the other terms in \eqref{bachg} we conclude that
\begin{align}
\begin{split}
B(g_{0}+h)&=B^{\prime}_{g_{0}}(h)+ (g_0 + h)^{-1} * h * \nabla^4_{g_0} h \\
&+ \sum_{j=2}^{6}\sum_{\alpha_{1}+\ldots+\alpha_{j}=4}(g_{0}+h)^{-j}*\nabla_{g_{0}}^{\alpha_{1}}h*\ldots*\nabla_{g_{0}}^{\alpha_{j}}h,
\end{split}
\end{align} 
For any $k\ge 1$, a similar argument shows that 
\begin{align}
\begin{split} \label{eqrem9}
\Delta^{k-1}_{g_{0}+h}B(g_0+h)&=\Delta_{g_0}^{k-1}B^{\prime}_{g_{0}}(h) 
+ (g_0 +h)^{-1} *h * \nabla^{2(k+1)}_{g_0}h \\
&+\sum_{j=2}^{\mathcal{I}_{k}}(g_{0}+h)^{-j}*\left(\sum_{\alpha_{1}+\ldots+\alpha_{j}=2(k+1)}
\nabla_{g_{0}}^{\alpha_{1}}h*\ldots*\nabla_{g_{0}}^{\alpha_{j}}h\right).
\end{split}
\end{align}
where $\mathcal{I}_{k}$ is some integer with $\mathcal{I}_{k} > 2(k+1)$.

Next, using scaling arguments it is clear that the terms 
$l.o.t.$ in (\ref{eqext2_1}) and (\ref{eqext2}) have the form
\begin{align}
\label{steqn}
\sum_{j=2}^{k+1}\left(\sum_{\alpha_{1}+\ldots+\alpha_{j}=2(k+1)-2j}
\nabla^{\alpha_{1}}Rm*\ldots*\nabla^{\alpha_{j}}Rm\right),
\end{align}
see for example the proof of Theorem 2.1 in \cite{g2}.
Equation (\ref{eqrem10}) follows 
from combining  \eqref{eqext2_1} or (\ref{eqext2}) in Subsection \ref{ext} 
with (\ref{eqrem9}), since the terms in \eqref{steqn} 
admit a similar expansion as the nonlinear terms in \eqref{eqrem9} (and
they do not affect the linearization). 
\end{proof}
Next, defining 
\begin{align}
c_{n,k}=\left\{
\begin{array}{ll}
\frac{1}{(4-n)(6-n)\ldots(2k-n)}&\mbox{if}~1\le k\le\frac{n}{2}-2\\
1& \mbox{if}~ k=\frac{n}{2}-1,
\end{array}
\right.
\end{align}
from Proposition \ref{proprem1}, we may write 
\begin{align}
\Omega^{(k)}(g_{0}+h) = c_{n,k} \Delta^{k-1} 
B_{g_{0}}^{\prime}(h)+F^{(k)}(h,g_{0}),\label{eqrem10_2}
\end{align}
where $F^{(k)}(h,g_{0})$ is the remainder in (\ref{eqrem10}). 
For the scalar curvature we will write 
\begin{align}
R(g_{0}+h) =R^{\prime}_{g_{0}}(h)+F^{\prime}(h,g_{0}),\label{eqrem10_4}
\end{align}
where $F^{\prime}(h,g_{0})$ is the error term in (\ref{eqrem11_1}). From now on, 
we will use $\nabla$ to denote $\nabla_{g_{0}}$, therefore all operators 
$\Delta$, $\delta$, $\tr$ are taken with respect to $g_{0}$. 
\subsection{Scalar-flat condition and the modified equation}\label{sfm}
In order to address the difficulty of not initially being able to prescribe $h$ to be 
divergence-free, 
we follow \cite{ct} and introduce a \emph{modified divergence operator} given 
by 
\begin{align}\label{deltat}
\delta_{t}h=\delta h-ti_{r^{-1}\frac{\partial}{\partial r}}h,
\end{align}
and we will 
show in Section \ref{df} that we can find a gauge where $\delta_{t}h=0$. 
A difference with the approach in \cite{ct} is that the obstruction-flat 
systems are \emph{not} elliptic even if we are able to prescribe $\delta_{t}h=0$ 
because $\Delta^{k-1}B^{\prime}_{g_{0}}(h)$ is traceless 
regardless of the gauge condition. Note that if $\delta_{t}h=0$ we obtain
\begin{align}
\begin{split}\label{eqsfm1}
\Delta^{k-1}B^{\prime}_{g_{0}}(h)=\Delta^{k-1}&\left(-\frac{1}{2(n-2)}\Delta^{2}h
-\frac{1}{2(n-1)(n-2)}\nabla^{2}\Delta\tr(h)\right.\\
&\left.-\frac{t}{2(n-1)}\nabla^{2}\delta i_{r^{-1}\frac{\partial}{\partial r}} h
-\frac{t}{(n-2)}\Delta\delta^{*}i_{r^{-1}\frac{\partial}{\partial r}}h\right.\\
&\left.+\frac{1}{2(n-2)(n-1)}\left(\Delta^{2}\tr(h)
-t\Delta\delta\left(i_{r^{-1}\frac{\partial}{\partial r}}h\right)\right)g_{0}
\right).
\end{split}
\end{align}
\noindent
At this point we use the scalar-flat condition on $g_{0}+h$. Assuming again 
that $\delta_{t}h=0$, the linearization of the scalar curvature at $g_{0}$ becomes
\begin{align}
\label{eqsfm2}
R^{\prime}_{g_{0}}(h)=  -\Delta\tr(h) + \delta \delta h = -\Delta\tr(h)
+t\delta i_{r^{-1}\frac{\partial}{\partial r}}h,
\end{align}
so from the scalar-flat equation we have
\begin{align}
\label{eqsfm3}
\Delta\tr(h)=t\delta i_{r^{-1}\frac{\partial}{\partial r}}h+F^{\prime}(h,g_{0}),
\end{align}
where $F^{\prime}(h,g_{0})$ is the remainder in (\ref{eqrem10_4}). 
Inserting (\ref{eqsfm3}) into (\ref{eqsfm1}) we obtain
\begin{align}
\begin{split} \label{eqsfm4}
\Delta^{k-1}B^{\prime}_{g_{0}}(h)=\Delta^{k-1}&\left(-\frac{1}{2(n-2)}\Delta^{2}h
-\frac{t}{2(n-2)}\nabla^{2}\delta i_{r^{-1}\frac{\partial}{\partial r}}h\right.\\
&\left.-\frac{t}{(n-2)}\Delta\delta^{*} i_{r^{-1}\frac{\partial}{\partial r}}h\right)
+\mathcal{E}^{(k)}(h,g_{0}),
\end{split}
\end{align}
where $\mathcal{E}^{(k)}(h,g_{0})$ is given by
\begin{align}
\label{eqsfm5}
\mathcal{E}^{(k)}(h,g_{0})
=\frac{1}{2(n-1)(n-2)} \Big( -\Delta^{k-1}\nabla^{2}F^{\prime}(h,g_{0})
+ \Delta^{k}F^{\prime}(h,g_{0})g_{0} \Big).
\end{align}
We now define a linear operator $\mathcal{P}^{(k)}_{t}$ by 
\begin{align}
\label{eqsfm6}
\mathcal{P}^{(k)}_{t}(h)=\frac{c_{n,k}}{n-2}
\Delta^{k-1}\left(-\frac{1}{2}\Delta^{2}h
-\frac{t}{2}\nabla^{2}\delta i_{r^{-1}\frac{\partial}{\partial r}}h
- t \Delta\delta^{*}i_{r^{-1}\frac{\partial}{\partial r}}h\right).
\end{align}
Clearly, the operator $\mathcal{P}^{(k)}_{t}$ is strictly elliptic.
From (\ref{eqrem10})
and (\ref{eqsfm4}), if $\delta_{t}h=0$ and  $g_{0}+h$ 
is scalar-flat, the (extended) obstruction-flat 
system may be written as
\begin{align}
0=\mathcal{P}^{(k)}_{t}h
+\mathcal{R}^{(k)}(h,g_{0})~\mbox{for}~1\le k\le\frac{n}{2}-1,\label{eqsfm7}
\end{align} 
where  
\begin{align*}
\mathcal{R}^{(k)}(h,g_{0})=c_{n,k}\mathcal{E}^{(k)}(h,g_{0})+F^{(k)}(h,g_{0}).
\end{align*}
Writing $\mathcal{E}^{(k)}(h,g_{0})$ schematically as 
$\nabla^{2k}F^{\prime}(h,g_{0})$, we easily see using 
the proof of Proposition \ref{proprem1} that $\mathcal{R}^{(k)}(h,g_{0})$ has 
the same form as the remainder in \eqref{eqrem10}
\begin{align}
\begin{split}\label{eqsfm7_1}
& \mathcal{R}^{(k)}(h,g_{0}) =(g_{0}+h)^{-1}*h*\nabla_{g_{0}}^{2(k+1)}h\\
&+\sum_{j=2}^{\mathcal{I}_k}(g_{0}+h)^{-j}*\left(\sum_{\alpha_{1}+\ldots+\alpha_{j}=2(k+1)}
\nabla^{\alpha_{1}}_{g_{0}}h*\ldots *\nabla^{\alpha_{j}}_{g_{0}}h\right),
\end{split}
\end{align}
so that none of the terms in $\mathcal{R}^{(k)}(h,g_{0})$ contribute to the 
linearization. This shows that 
(\ref{eqsfm7}) defines a family of elliptic equations on 
$\R^{n}\backslash\{0\}$. This same argument applies equally to the system 
\eqref{eqsrt11in}. We have proved
\begin{corollary}\label{corsfm1}
Let $1\le k\le \frac{n}{2}-1$. If $g=g_{0}+h$ is scalar-flat and $\Omega^{(k)}$-flat 
with $\delta_{t}h=0$, then $h$ satisfies
\begin{align}
\mathcal{P}^{(k)}_{t}h+\mathcal{R}^{(k)}(h,g_{0})
=0,
\label{eqsfm10}
\end{align}
where $\mathcal{P}_{t}^{(k)}$ is the linear operator given by (\ref{eqsfm6}) and $\mathcal{R}^{(k)}(h,g_{0})$ is the error term given by \eqref{eqsfm7_1}. 
The operator $\mathcal{P}^{(k)}_{t}$  is strictly elliptic.

 Let $k = 1$ if $n = 3$, or 
$1\le k \le\frac{n}{2}-1$ if $n \geq 4$.
If $g=g_{0}+h$ is scalar-flat and solves \eqref{eqsrt11in} with $\delta_{t}h=0$, 
then $h$ also satisfies \eqref{eqsfm10} with a remainder term of the same form.   
\end{corollary}

\section{Weighted H\"{o}lder and Sobolev Spaces}
In this section we introduce weighted spaces that will be useful
in the analysis needed to construct divergence-free gauges.
We start by 
reviewing some of the notation in \cite{ct}. 
If we write $\R^{n}$ as $\R^{n}=C(S^{n-1})$, we define maps 
$\psi_{a}:C(S^{n-1})\rightarrow C(S^{n-1})$ for $a>0$ given by 
$\psi_{a}(r,x)=(ar,x)$.
The weighted H\"{o}lder norms are defined as follows: 
if $u>0$, let $A_{u}$ be the natural scaling on tensors of type $(p,q)$, i.e.
\begin{align}
A_{u}=\underbrace{(\psi_{u^{-1}})_{*}\otimes\cdots\otimes(\psi_{u^{-1}})_{*}}_{p}\otimes\underbrace{\psi_{u}^{*}\otimes\cdots\otimes\psi_{u}^{*}}_{q}.\label{eqwh1}
\end{align}
Given a tensor $T$ of type $(p,q)$ we have
\begin{align}
|T_{(u,x)}|_{m,\alpha;0}=|u^{p-q}A_{u}T_{(1,x)}|_{m,\alpha;0},\label{eqwh2}
\end{align}
where $|\cdot|_{m,\alpha;0}$ is the $C^{m,\alpha}$-norm with 
respect to the flat metric $g$ at the point $(1,x)$. We can now define H\"{o}lder norms by 
\begin{align}
|T|_{m,\alpha;0}=\sup_{(u,x)}|T_{(u,x)}|_{m,\alpha;0},\label{eqwh3}
\end{align}
and also weighted H\"{o}lder norms given by
\begin{align}
|T|_{m,\alpha;l}=|r^{-l}T|_{m,\alpha;0},\label{eqwh4}
\end{align}
for any $l \in \mathbb{R}$. 
We say that a tensor $T$ of type $(p,q)$ is in $\ten^{p,q}_{m,\alpha;l}$ 
if $|T|_{m,\alpha;l}<\infty$. 
We use $A_{c,d}(0)$ to denote the annulus 
$A_{c,d}(0)=\{(r,x)|c<r<d\}$. 
For a tensor $h$ of type $(u,v)$ and 
$\delta\in\R$ we define a weighted $L^{p}$ norm by
\begin{align}
\|h\|_{{L^{\prime}}^{p,u,v}_{\delta}}
=\left(\int_{\R^{n}}|h|^{p}|x|^{-\delta p-n}dx\right)^{\frac{1}{p}},\label{eqws1}
\end{align}
where $|\cdot|$ is the usual pointwise norm on tensors of type $(u,v)$. 
For any nonnegative integer $m$, we also define weighted Sobolev norms by
\begin{align}
\|h\|_{{W^{\prime}}^{m,p,u,v}_{\delta}}
=\sum_{j=0}^{m}\|\nabla^{j}h\|_{{L^{\prime}}^{p,u,v}_{\delta-j}},\label{eqws2}
\end{align}
and then ${W^{\prime}}^{m,p,u,v}_{\delta}$ is the space 
\begin{align}
{W^{\prime}}^{m,p,u,v}_{\delta}
=\left\{h:\|h\|_{{W^{\prime}}^{m,p,u,v}_{\delta}}<\infty\right\}.\label{eqws2_1}
\end{align}
For further properties of these weighted Sobolev spaces see~\cite{bar}. 

For notational convenience, if $h$ is a symmetric $(0,2)$-tensor 
we will use $|||h|||_{a,b}$ to denote 
the norm $\|h\|_{L^{2,0,2}_{0}(A_{a,b}(0))}$, i.e., the $L^{2}$ norm of $h$ with 
weight $0$ on the annulus $A_{a,b}(0)$. Another way to construct the norm 
$|||\cdot|||$ is as follows: consider the weighted inner product on the 
slices $(r,S^{n-1})$ given by
\begin{align}
\langle\langle h_{1},h_{2}\rangle\rangle=r^{-(n-1)}\int_{(r,S^{n-1})}
\langle h_{1},h_{2}\rangle dV_{g_{S^{n-1}(r)}},\label{eqdble}
\end{align}
where $\langle\cdot,\cdot,\rangle$ is the usual pointwise inner product. It follows that
\begin{align}
|||h|||^{2}_{a,b}=\int_{a}^{b}r^{-1}\|h\|^{2}dr,
\end{align}
where $\|\cdot\|$ is the norm defined by the inner product 
$\langle\langle\cdot,\cdot\rangle\rangle$ in (\ref{eqdble}). 
The norms $|||\cdot|||$ are scale invariant in the sense that 
if $h$ is a $(0,2)$ tensor and if we let $q=a^{-2}\psi_{a}^{*}h$ then 
\begin{align}
|||h|||_{a,aL}=|||q|||_{1,L}.
\end{align}
Finally, we say that 
a $(p,q)$-tensor $T$ is {\em{radially parallel}} if 
\begin{align}
\nabla_{\frac{\partial}{\partial r}} T = 0. 
\end{align}
\subsection{Divergence and the Lie derivative operator}\label{divlie}
We now consider the operator $\Box:\Lambda^{1}(\R^{n})\rightarrow\Lambda^{1}(\R^{n})$ 
defined by $\Box \xi=\delta L_{\xi}g_{0}$, where $L_\xi$ is the 
Lie derivative operator. This operator is 
formally self-adjoint and elliptic. From now on we use 
 $\tilde{\Delta}_{H}$, $\tilde{\tr}$,$\tilde{\delta}$ and $\tilde{d}$ to denote the Hodge laplacian, trace, divergence and exterior differentiation  in the cross section metric $g_{S^{n-1}}$ respectively. Following  \cite[Section 2]{ct}, if we write $\xi$ in polar coordinates as
\begin{align}\label{1formpolar}
\xi=fdr+\omega,
\end{align}
we have
\begin{align}\label{liederivative}
L_{\xi}g_{0}=2f^{\prime}dr\otimes dr+\left(\omega^{\prime}-2r^{-1}\omega
+\tilde{d}f\right)\boxtimes dr+L_{\omega}g_{S^{n-1}}+2rfg_{S^{n-1}},
\end{align}
here we use primes to denote differentiation respect to $r$. Also, given a 1-form $\alpha$ 
we denote by $\alpha\boxtimes dr$ the symmetric product $\alpha\otimes dr+dr\otimes\alpha$. 
If now $h$ is a symmetric $(0,2)$-tensor written in polar coordinates  as
\begin{align}
h=h_{00}dr\otimes dr+\alpha\boxtimes dr+B,
\end{align}
where $B$ is a symmetric $(0,2)$ tensor whose radial components are zero, the divergence of $h$ is given by
\begin{align}
\begin{split}
\label{divergence}
\delta h &= \left( h^{\prime}_{00} +(n-1)r^{-1}h_{00}+r^{-2}\tdel\alpha
-r^{-3}\tilde{\tr}(B) \right) dr\\
& \ \ \ \ \ \ + \alpha^{\prime}+ (n-1)r^{-1}\alpha+r^{-2}\tdel B.
\end{split}
\end{align}
Combining (\ref{liederivative}) and (\ref{divergence}), the operator $\Box$ takes the form
\begin{align}
\begin{split}
\label{boxpolar}
\Box \xi&=\left(2f^{\prime\prime}+2(n-1)r^{-1}f^{\prime}+r^{-2}\left(-\tlaph f -2(n-1)f\right)+r^{-2}\tdel\omega^{\prime}-4r^{-3}\tdel\omega\right)dr\\
& \ \ \ \ \ \ + \omega^{\prime\prime}+(n-3)r^{-1}\omega^{\prime}
+r^{-2}\left(-\tlaph\omega+\td\tdel\omega\right)+\td f^{\prime}+r^{-1}(n+1)\td f.
\end{split}
\end{align}
As pointed out in \cite[Section 2]{ct}, 
any 1-form defined on $\R^{n}\backslash\{0\}$ can be written as 
an infinite sum of forms of two types
\begin{enumerate}
\item Type I: $p(r)\psi$ where $\psi\in\Lambda^{1}(T^*S^{n-1})$, 
$\tilde{\delta}\psi=0$ and $\tilde{\Delta}_{H}\psi=\mu\psi$,
\item Type II: $r^{-1}l(r)\phi dr+u(r)r\tilde{d}\phi$ where 
$\phi\in\Lambda^{0}(T^*S^{n-1})$ and $\tilde{\Delta}_{H}\phi=\nu\phi$. 
\end{enumerate}
Moreover, the operator $\Box$ preserves these two types of forms. If $\xi$ is a 1-form of type I or II, the equation $\Box\xi=0$ reduces to solving a second order linear system of ordinary differential equations of at most two equations. In order to see what these systems look like we consider the change of variable $r=e^{s}$ and use $p(s)$ to denote $p(e^{s})$ for forms of type I and we use $l(s)$, $u(s)$  to denote $l(e^{s})$ and $u(e^{s})$ respectively for forms of type II. With this notation we have for example
\begin{align}
p^{\prime}=e^{-s}\partial_{s}p,~p^{\prime\prime}=e^{-2s}\left(\partial^{2}_{s}p-\partial_{s}p\right).
\end{align} 
On forms of type I, $\Box\xi$ is given by
\begin{align}\label{boxtype1}
\Box\xi=e^{-2s}\left(\partial^{2}_{s}p+(n-4)\partial_{s}p-\mu p\right)\psi,
\end{align}   
while on forms of type II, $\Box\xi$ takes the form
\begin{align}
\begin{split}
\label{boxtype2}
\Box\xi&=e^{-2s}\left(2\partial_{s}^{2}l+2(n-4)\partial_{s}l-4(n-2+\frac{\nu}{4})l-\nu\partial_{s}u+4\nu u\right)\phi  ds\\
& \ \ \ \ \ \ \ + e^{-2s}\left(\partial^{2}_{s}u+(n-4)\partial_{s}u-2\nu u+\partial_{s}l+nl\right)\td\phi.
\end{split}
\end{align}
The solutions of (\ref{boxtype1}) and (\ref{boxtype2}) are given by the following  
\begin{align}
\label{eqws5}
\alpha = \frac{4-n}{2}, \
\theta=\sqrt{\alpha^{2}+\mu}, \ 
a^{\pm}=\alpha\pm\theta.
\end{align}
All solutions of $\Box \xi=0$ of type I are given by
\begin{align}
\xi=r^{a^{\pm}}\psi.\label{eqws6}
\end{align}
In this case, since $r\psi$ is radially parallel, we see that 
the order of growth of $\xi$ is $a^{\pm}-1$. 
For solutions $\xi$ of type II, we set
\begin{align}\label{eqws9}
\beta=\frac{2-n}{2}, \ 
\omega=\sqrt{\beta^{2}+\nu}, \
b^{\pm}=\beta\pm\omega,
\end{align}
and then $\xi$ is either of the form
\begin{align}
r^{b^{\pm}}\tilde{d}\phi+b^{\pm}r^{b^{\pm}-1}\phi dr,\label{eqws10}
\end{align} 
or of the form
\begin{align}
2r^{b^{\pm}+2}\tilde{d}\phi+b^{\mp}r^{b^{\pm}+1}\phi dr.\label{eqws11}
\end{align}
See \cite[Section 2]{ct} for more details. The above computations 
motivate the following definition
\begin{definition}\label{defws1}{\em
The set $E$ of all numbers $a^{\pm}-1$, $b^{\pm}-1$ or $b^{\pm}+1$ with $a^{\pm}$, 
$b^{\pm}$ defined by \eqref{eqws5}, \eqref{eqws9}, is called the set of 
\emph{exceptional values for} $\Box$. If $\gamma\in\R\backslash E$ then $\gamma$ 
is said to be \emph{nonexceptional}.}
\end{definition}

\begin{remark}{\em
All elements in $E$ are integers, in fact, computing the eigenvalues 
of $\tilde{\Delta}_{H}$ on 1-forms in $S^{n-1}$ as in 
\cite[Theorem C]{foll}, one can prove that all numbers in (\ref{eqws6}) have the form
\begin{align}\label{exceptional1}
a^{\pm}_{j}=-\frac{(n-4)}{2}\pm\frac{1}{2}\left(n-2+2j\right)~\mbox{for}~j=1,2,\ldots,
\end{align}
and all numbers in \eqref{eqws9} have the form
\begin{align}\label{exceptional2}
b^{\pm}_{j}=-\frac{(n-2)}{2}\pm\frac{1}{2}\left(n-2+2j\right)~\mbox{for}~j=0,1,2,\ldots.
\end{align}
}
\end{remark}
An important property of $\Box$ is

\begin{proposition}\label{propws1} On $\mathbb{R}^n \setminus \{0\}$, 
if $\gamma$ is nonexceptional then 
$\Box:{W^{\prime}}^{2,p,0,1}_{\gamma}\rightarrow {W^{\prime}}^{0,p,0,1}_{\gamma-2}$ 
is an isomorphism with bounded inverse.
\end{proposition}
\begin{proof} Compare \cite[Theorem 1.7]{bar}.
\end{proof}
Finally, note that from (\ref{exceptional1}) and (\ref{exceptional2}) it follows that $1$ is an exceptional value for $\Box$, which means that there are elements in the kernel of $\Box$ with linear growth, i.e., forms $\xi$ with $\Box\xi=0$ satisfying $\xi=r\eta$ where $\eta$ is radially parallel.

All 1-forms of type I in the kernel of $\Box$ which have linear growth have the form 
\begin{align}\label{killing}
\xi=r^{2}\psi,
\end{align} 
where $\psi$ is dual to a Killing field in $S^{n-1}$. For forms of type II, all solutions of $\Box\xi=0$ that have linear growth correspond to the eigenvalues $\nu=0,2n$ of $\tilde{\Delta}_{H}$ on functions, moreover, in that case the solution corresponding to $\nu=0$ is 
\begin{align}\label{confkilling}
\xi=rdr,
\end{align}
and the solution corresponding to $\nu=2n$ is
\begin{align}\label{sph2}
\xi=2r\phi\tilde dr+r^{2}\tilde{d}\phi.
\end{align}
Note that the forms in (\ref{killing}) and (\ref{sph2}) have linear growth because $r\psi$ and $r\tilde{d}\phi$ are radially parallel.

\subsection{A modified $\Box$ operator}\label{modbox}

Given $t\ne 0$, let $\Box_{t}$ be the modified operator 
\begin{align}
\Box_{t}\xi \equiv \delta_t L_{\xi}g_{0} 
=\Box \xi -ti_{r^{-1}\frac{\partial}{\partial r}}L_{\xi}g_{0}. 
\end{align}
In order to compute $\Box_{t}$ for $\xi$ as in (\ref{1formpolar}) we start 
by noting that 
\begin{align*}
i_{r^{-1}\frac{\partial}{\partial r}}L_{\xi}g_{0}=2r^{-1}f^{\prime}dr +r^{-1}\left(\omega^{\prime}-2r^{-1}\omega+\td f\right).
\end{align*}
The modified operator $\Box_{t}$ is then computed to be
\begin{align*}
\Box_{t}\xi&=
\left(2f^{\prime\prime}+2(n-1-t)r^{-1}f^{\prime}+r^{-2}\left(-\tlaph f-2(n-1)f\right)\right)dr\\
& \ \ \ \ + r^{-2}\left(\tdel\omega^{\prime}-4r^{-3}\tdel\omega\right)dr+
\omega^{\prime\prime}+(n-3-t)r^{-1}\omega^{\prime}\\
& \ \ \ \ +r^{-2}\left(-\tlaph\omega
+2t\omega+\td\tdel\omega\right)+
\td f^{\prime}+r^{-1}(n+1-t)\td f.
\end{align*}
Using again the change of variable $r=e^{s}$ and the notation in Section \ref{divlie}, we see that $\Box_{t}\xi$ for $\xi$ a 1-form of type I is given by
\begin{align}\label{boxttype1}
\Box_{t}\xi=e^{-2s}\left(\partial^{2}_{s}p+(n-4-t)\partial_{s}p-(\mu-2t)p\right)\psi.
\end{align}   
On forms of type II, $\Box_{t}\xi$ is given by
\begin{align}
\begin{split}
\label{boxttype2}
\Box_{t}\xi&=e^{-2s}\left(2\partial_{s}^{2}l+2(n-4-t)\partial_{s}l-4(n-2-\frac{t}{2}+\frac{\nu}{4})l-\nu\partial_{s}u+4\nu u\right)\phi ds\\
&\ \ \ \ \ +e^{-2s}\left(\partial^{2}_{s}u+(n-4-t)\partial_{s}u-2(\nu-t)u+\partial_{s}l+(n-t)l\right)\td\phi.
\end{split}
\end{align}
We conclude that in these cases, the system 
\begin{align}\label{kernelboxt}
\Box_{t}\xi=0,
\end{align}
reduces again to a constant coefficient system of ordinary differential equations. 
As in the discussion at the end of Subsection \ref{divlie}, 
we are interested in solutions of (\ref{kernelboxt}) which are \emph{essentially linear}, 
i.e., solutions $\xi$ that satisfy for every $0<\gamma<1$,
\begin{align}\label{lineargrowth}
\xi=O(r^{1+\gamma})~\text{as}~r\rightarrow\infty,~\xi=O(r^{1-\gamma})~\text{as}~r\rightarrow 0.
\end{align}
For these solutions we have
\begin{proposition}\label{nolineargrowth}
There exists $t_{0}>0$ such that for $0<|t|<t_{0}$ there is  a number $\gamma_{0}(t)$ with $0<\gamma_{0}(t)<1$ such that if $\xi$ is a 1-form in the kernel of $\Box_{t}$ satisfying \eqref{lineargrowth} for some $0<\gamma<\gamma_{0}(t)$ then $\xi$  is dual to a Killing vector field in $\R^{n}$.  
\end{proposition}
\begin{proof}
Since $\Box_{t}\xi=\delta_{t}L_{\xi}g_{0}$, 
it follows that a 1-form which is dual to a Killing field in $\R^{n}$ is 
in the kernel of $\Box_{t}$ \emph{regardless} of the gauge condition. Next, 
since the general solution may be written as an infinite sum of
$1$-forms of Type I and Type II, 
it suffices to prove the proposition for $1$-forms of either type. 
When $t=0$ and as 
pointed out in Section \ref{divlie}, all solutions of (\ref{kernelboxt}) in 
separated variables which satisfy (\ref{lineargrowth}) correspond to the 
eigenvalues $\mu=2(n-2)$ on forms of type I and $\nu=0,2n$ on forms of type~II. 
For $t \neq 0$ small, 
the growth rates 
are small perturbations 
of the rates $a^{\pm}_{j}-1$ and $b^{\pm}_{j}-1$ with $a^{\pm}_{j}$, 
$b^{\pm}_{j}$ given by (\ref{exceptional1}) and (\ref{exceptional2}), 
moreover the rates corresponding to eigenvalues
$\mu > 2n$ or $\nu > 2n$ 
have real parts bounded away from $1$ and hence, for these 
solutions the proposition follows. It only remains to consider the kernel 
corresponding to the eigenvalues $\mu=2(n-2)$ and $\nu=0,2n$. 
From \cite[Theorem~C]{foll}, all eigenvalues $\mu$ corresponding to 1-forms of type I are
\begin{align}
\mu=(j+1)(j+n-3)~\text{for}~j=1,2,\ldots,
\end{align}
and then from (\ref{boxttype1}) all solutions of $\Box_{t}\xi=0$ with $\xi$ 
a 1-form of type I can be written as
\begin{align}
\xi=r^{c^{\pm}_{j}(t)-1}\left(r\psi_{j}\right),
\end{align}
where 
\begin{align}\label{ratetype1}
c^{\pm}_{j}(t)=\frac{-(n-4-t)\pm\sqrt{(n-2+2j)^{2}-2nt+t^{2}}}{2},
\end{align}
and $\psi_{j}$ is a 1-form with eigenvalue $(j+1)(j+n-3)$. It follows that if 
$t\ne 0$ is sufficiently small, all solutions given by (\ref{ratetype1}) are such that $Re\left(c_{j}^{\pm}(t))-1\right)$ is bounded away from 1 except for $c_{1}^{+}(t)-1$ 
which equals $1$ for \emph{any} $t$. In this case 
$r^2 \psi_{1}$ is dual to a Killing field in 
$\R^{n}$. For those 1-forms of type II corresponding to the eigenvalues $\nu=0,2n,$
the growth rates are strictly bounded away from $1$ for $t \neq 0$ sufficiently
small; the proof is similar and is omitted. 
\end{proof}

\section{Existence of divergence-free gauges}
\label{df}
In order to construct divergence-free gauges, we use the approach in \cite{ct} 
which consists in using the ALE of order 0 condition to initially prove that 
we can fix a gauge such that the modified divergence-free condition 
$\delta_{t}h=0$ for $t\ne 0$ is satisfied.
One is then interested in the $\delta_{t}$-free kernel of the modified 
operator $\mathcal{P}_{t}^{(k)}$. 
This kernel can be classified into three types:
\begin{enumerate}
  \item\label{class1} Growth solutions, i.e., solutions that are $O(r^{\beta^{\prime}})$ on $\R^{n}\backslash\{0\}$ for some $\beta^{\prime}>0$, 
  \item\label{class2} Decay solutions, i.e., solutions that are $O(r^{-\beta^{\prime}})$ on $\R^{n}\backslash\{0\}$ for some $\beta^{\prime}>0$, 
  \item\label{class3} ``Degenerate" solutions, i.e., solutions that are $O(r^{\gamma})$ as $r\rightarrow\infty$ and  $O(r^{-\gamma})$ as $r\rightarrow 0$ for some $\gamma > 0$
sufficiently small (depending upon $t$).  
  \end{enumerate}  The main step is to prove that solutions with the 
behavior described in (\ref{class3}) do \emph{not} occur 
(see Proposition \ref{propmod1}).   
We can then prove a growth estimate for solutions of the 
linear elliptic equation $\mathcal{P}_{t}^{(k)}h=0$ 
(see Lemma \ref{lemdf2}) which we call the \emph{Three Annulus Lemma}. 
The next step is to use scaling properties 
of the nonlinear system (\ref{eqsfm10}), and elliptic estimates to prove a 
nonlinear version of the Three Annulus Lemma (Lemma \ref{lemdf3}), so that the
behavior of solutions of (\ref{eqsfm10}) can be modeled after the behavior of 
solutions of the linearized equation. 
Consequently,  
we can use the nonlinear Three Annulus Lemma and the ALE of order zero 
condition to rule out 
the behavior described in (\ref{class1}) and \eqref{class3}
for solutions of the nonlinear equation. 
It follows that the only 
possible behavior at infinity for solutions of the nonlinear equation in the 
$\delta_{t}$-gauge is that of decay solutions in (\ref{class2}), 
which yields a gauge where the metric $g$ is ALE of positive order 
(see Corollary \ref{corgdf1}). With this improvement, one we can 
easily construct a global divergence-free gauge, see Proposition \ref{dftau}.   

The Three Annulus Lemma
was introduced in  \cite{ls1} and used in \cite{ct} for the 
Ricci-flat case. Even though 
our statement of the Three Annulus Lemma is very similar to that of \cite{ct}, 
we base our proof on a result called Turan's Lemma that we discuss in 
Appendix \ref{turan_lemma}. Our
case is complicated by the fact that higher
powers of $\log$ may enter into the asymptotic expansions since 
the system is of higher order. 

\subsection{The linearized equation in separated variables}
\label{lesv}
We consider solutions on $\R^{n}\backslash\{0\}$ of the system
\begin{align}\label{eqdf3_5}
\mathcal{P}_{t}^{(k)}h=0,
\end{align}
for $1\le k\le \frac{n}{2}-1$ if $n \geq 4$, or for $k = 1$ if $n=3$,
with $\mathcal{P}^{(k)}_{t}$ given by (\ref{eqsfm6}). 
With notation as in \cite{ct}, we write the general solution of $\mathcal{P}^{(k)}_{t}h=0$ 
as an infinite sum of the form
\begin{align}
h=\sum_{j=0}^{\infty}\left(l_{j}\phi_{j}dr\otimes dr+k_{j}\tau_{j}\boxtimes dr+f_{j}B_{j}+p_{j}\phi_{j}r^{2}\gS\right),\label{L2expansion}
\end{align}
where 
\begin{enumerate}
\item The functions $l_{j}$,$k_{j}$, $f_{j}$ and $p_{j}$ are radial, and 
the $(0,2)$ tensors $\phi_{j}dr\otimes dr$, $\tau_{j}\boxtimes dr$, $B_{j}$, 
and $\phi_{j}r^{2}\gS$ are radially parallel.
\item The components $\phi_{j}\in\Lambda^{0}(T^{*}S^{n-1}) $, $\tau_{j}\in\Lambda^{1}(T^{*}S^{n-1})$, $B_{j}\in S_{0}^{2}(T^{*}S^{n-1})$ are eigenfunctions of the rough laplacian. Here $S_{0}^{2}(T^{*}S^{n-1})$ is the traceless component of $S^{2}(T^{*}S^{n-1})$.
\item The set
\begin{align*}
\{l_{j}\phi_{j}dr\otimes dr+k_{j}\tau_{j}\boxtimes dr+f_{j}B_{j}+p_{j}\phi_{j}r^{2}\gS\}_{j},
\end{align*}
is orthogonal in the inner product $\langle\langle \cdot , \cdot \rangle\rangle$ 
in (\ref{eqdble}). 
\end{enumerate}
It is clear that $\mathcal{P}^{(k)}_{t}$ preserves the 
expansion \eqref{L2expansion} in the sense that
\begin{align*}
&\mathcal{P}^{(k)}_{t}\left(l_{j}\phi_{j}dr\otimes dr+k_{j}\tau_{j}\boxtimes dr+f_{j}B_{j}+p_{j}\phi_{j}\gS\right)\\
&=F_{j}^{(1)}\phi_{j}dr\otimes dr+F_{j}^{(2)}\tau_{j}\boxtimes dr+F^{(3)}_{j}B_{j}+F^{(4)}_{j}\phi_{j}r^{2}\gS,
\end{align*}
where $F^{(c)}_{j}=F^{(c)}_{j}(l_{j},k_{j},f_{j},p_{j})$ for $c=1,2,3,4$. It follows that if $\mathcal{P}_{t}^{(k)}h=0$ then for each $j$ we must have
\begin{align}
\mathcal{P}^{(k)}_{t}\left(l_{j}\phi_{j}dr\otimes dr+k_{j}\tau_{j}\boxtimes dr+f_{j}B_{j}+p_{j}\phi_{j}r^{2}\gS\right)=0,\label{ptode}
\end{align}
and 
\begin{align}\label{ptsystem}
F^{(c)}_{j}(l_{j},k_{j},f_{j},p_{j})\equiv 0,~\text{for}~c=1,2,3,4.
\end{align} 
The system (\ref{ptsystem}) is a linear system of ordinary differential equations which 
is homogeneous of order $2(k+1)$. Using the change of variable $r=e^{s}$, (\ref{ptsystem}) can be written as a constant coefficient linear system of ordinary differential equations, in particular, for every $j=0,1,\ldots,$  the system (\ref{ptsystem}) reduces to a first order constant coefficient linear system of the form
\begin{align}\label{firstorder}
\dot{X}=M_{j}X,
\end{align}
where $M_{j}$ is a matrix of order $8(k+1)\times 8(k+1)$. Let $\Phi_{j}$ is the characteristic polynomial of the matrix $A_{j}$ and suppose that we factor $\Phi_{j}$ as
\begin{align}\label{charpol}
\Phi_{j}(z)=\prod_{a=1}^{m_{j}}(z-\zeta_{j,a})^{n_{j,a}},
\end{align}
with 
\begin{align}
\sum_{a=1}^{m_{j}}n_{j,a}=8(k+1).
\end{align}
Note that in general, the roots $\zeta_{j,a}$ depend on $t$ and $k$, however, for simplicity we omit this dependence in the notation above.
Each of the functions $l_{j}$, $k_{j}$, $f_{j}$, $p_{j}$ may be expressed 
as a linear combination of functions of the form
\begin{align}\label{gensol}
(\log(r))^{b}r^{\zeta_{j,a}}~\text{with}~b=0,\ldots, n_{j,a}-1.
\end{align}  
Fix $j$ and let $T^{(c)}_{j}$ with $c=1,2,3,4,$ be the $(0,2)$ tensors 
\begin{align}
\label{Ttensors}
T_{j}^{(1)}=\phi_{j}dr\otimes dr,~T_{j}^{(2)}=\tau_{j}\boxtimes dr,~T_{j}^{(3)}=B_{j},~T_{j}^{(4)}=\phi_{j}r^{2}\gS,
\end{align}
so that we can write
\begin{align}
l_{j}\phi_{j}dr\otimes dr+k_{j}\tau_{j}\boxtimes dr+f_{j}B_{j}+p_{j}\phi_{j}r^{2}\gS=\sum_{c=1}^{4}q_{j,c}T^{(c)}_{j},
\end{align}
where 
\begin{align}
q_{j,1}=l_{j},~ q_{j,2}=k_{j},~ q_{j,3}=f_{j}~ \text{and}~ q_{j,4}=p_{j}.
\end{align}
It follows that if $h$ satisfies $\mathcal{P}^{(k)}_{t}h=0$ on $\R^{n}\backslash\{0\}$ then $h$ can be expanded as an infinite sum of the form
\begin{align}\label{parseval}
h=\sum_{j=0}^{\infty}\sum_{c=1}^{4}q_{j,c}T_{j}^{(c)},
\end{align}
where
\begin{enumerate}
  \item\label{sepvar1} The  $(0,2)$-tensors $T_{j}^{(c)}$ are radially parallel.
  \item\label{sepvar2} The set $\{T^{(c)}_{j}\}_{j,c}$ is orthogonal with respect to the norm $\langle\langle \cdot , \cdot \rangle\rangle$.
  \item\label{sepvar3} The radial functions $q_{j,c}$ are linear combinations of functions of the form (\ref{gensol}).
  \end{enumerate}
Let us write the radial function $q_{j,c}(r)$ as
\begin{align}\label{lincomb}
q_{j,c}(r)=\sum_{a=1}^{m_{j}}\sum_{b=0}^{n_{j,a}-1}d_{a,b,c}(\log(r))^{b}r^{\zeta_{j,a}},
\end{align}
where $d_{a,b,c}$ are complex numbers. From (\ref{lincomb}) we introduce the following sets
\begin{align}
A^{+}_{j}&=\{1\le a\le m_{j}:Re(\zeta_{j,a})>0\},\label{Aplus}\\
A^{-}_{j}&=\{1\le a\le m_{j}:Re(\zeta_{j,a})<0\},\label{Aminus}\\
A^{0}_{j}&=\{1\le a\le m_{j}:Re(\zeta_{j,a})=0\}.\label{Azero}
\end{align}
We will use $q_{j,c}^{\pm}$ to denote
\begin{align}\label{growthdecaypart}
q_{j,c}^{\pm}(r)=\sum_{a\in A^{\pm}_{j}}\sum_{b=0}^{n_{j,a}-1}d_{a,b,c}(\log(r))^{b}r^{\zeta_{j,a}},
\end{align}
and $q_{j,c}^{0}$ will be used to denote
\begin{align}\label{degeneratepart}
q_{j,c}^{0}(r)=\sum_{a\in A^{0}_{j}}\sum_{b=0}^{n_{a}-1}d_{a,b,c}(\log(r))^{b}r^{\zeta_{j,a}}.
\end{align}
With this we have a decomposition of the form
\begin{align}
h=h^{+}+h^{-}+h^{0},\label{pos_neg_log}
\end{align}
where
\begin{align} \label{pos}
h^{\pm}=\sum_{j=0}^{\infty} \sum_{c=1}^4q_{j,c}^{\pm}T^{(c)}_{j},
\mbox{ and } h^{0}=\sum_{j=0}^{\infty} \sum_{c=1}^4 q_{j,c}^{0}T^{(c)}_{j}.
\end{align}
\begin{definition}
\label{degdef}
{\em
A solution $h$ of \eqref{gensol} is a {\em{degenerate solution of}} \eqref{eqdf3_5} 
if $h = h^0$.} 
\end{definition}  
It will also be important for us to consider for any nonnegative integer $j$ the number
\begin{align}\label{growthrate}
\beta_{j}=\min\{ |Re(\zeta_{j,a})|:a\in A^{\pm}_{j}\}.
\end{align}
\subsection{Estimates for the linearized equation}\label{ee}
Consider a solution of (\ref{eqdf3_5}) on an annulus, i.e., a solution of the problem
\begin{align}
\mathcal{P}^{(k)}_{t}h=0~\text{on}~A_{a,b}(0),\label{lin_eq_annulus}  
\end{align}
where $0<a<b$. Note that since $A_{a,b}(0)$ and $\R^{n}\backslash\{0\}$ have the same cross section, if $h$ is a solution of (\ref{lin_eq_annulus}) then we can repeat the analysis in 
Subsection \ref{lesv} to decompose $h$ as $h=h^{+}+h^{-}+h^{0}$. 
For solutions of (\ref{lin_eq_annulus}), however, the infinite sum (\ref{parseval}) may \emph{not} be defined outside of $A_{a,b}(0)$. 
By definition of the norm $|||\cdot|||$, if we expand a solution of $h$ of (\ref{lin_eq_annulus}) 
satisfying $|||h|||_{a,b}<\infty$ as in (\ref{parseval}) we see that
\begin{align}\label{L2norm}
|||h|||_{a,b}^{2}=\sum_{j=0}^{\infty}\sum_{c=1}^{4}\lambda_{j}^{(c)}\int_{a}^{b}|q_{j,c}(r)|^{2}r^{-1}dr,
\end{align}  
where 
\begin{align}
\lambda^{(c)}_{j}=\langle\langle T^{(c)}_{j},T^{(c)}_{j}\rangle\rangle.
\end{align}
We consider again the numbers $\beta_{j}$ in (\ref{growthrate}) and we define
\begin{align}
\beta=\inf_{j=0,1,\ldots}\{\beta_{j}\}.
\end{align}
The number $\beta$ is well-defined and positive for $t$ sufficiently 
small, since the equation \eqref{eqdf3_5} is a perturbation of 
$\Delta^{k +1}h = 0$, which has indicial roots contained in 
$\mathbb{Z} \subset \mathbb{C}$ (compare \eqref{hpex}-\eqref{eqlin6_2}).

We have the following property for solutions of (\ref{lin_eq_annulus}):
\begin{lemma}\label{posneg} For $t$ sufficiently small, 
let $0<\beta^{\prime}<\frac{1}{2}\beta$. Let $h$ be a solution of 
\eqref{lin_eq_annulus} on an annulus of the form $A_{a,L^{2}a}(0)$ where $a>0$ and $L>1$, 
and consider the decomposition 
\begin{align}
h=h^{+}+h^{-}+h_{0}~\text{on}~A_{a,L^{2}a}(0).\label{pos_neg_log_annulus}
\end{align}
Then there exists $L_{0}=L_{0}(\beta, \beta^{\prime})>1$ such that if $L>L_{0}$, then 
\begin{align}
|||h^{+}|||_{La,L^{2}a}\ge L^{\beta^{\prime}} |||h^{+}|||_{a,La},\label{annulusgrowth}
\end{align} 
and
\begin{align}
|||h^{-}|||_{La,L^{2}a}\le L^{-\beta^{\prime}} |||h^{-}|||_{a,La}.\label{annulusdecay}
\end{align} 
\end{lemma}
\begin{proof}
By the scale invariance of the norms $|||\cdot|||$, it suffices to prove the lemma for $a=1$. The proof is completed in Appendix \ref{turan_lemma} using Turan's Lemma. 
\end{proof}
We note that we are only able to prove (\ref{annulusgrowth}) and (\ref{annulusdecay}) for $0<\beta^{\prime}<\frac{1}{2}\beta$ and not for $0<\beta^{\prime}<\beta$ as in \cite{ct}. 
However, the estimates (\ref{annulusgrowth}) and (\ref{annulusdecay}) are
sufficient for our purpose. Next, we use this to prove
\begin{lemma}[Three Annulus Lemma]\label{lemdf2}
Let $a>0$, $L>1$ and suppose that $h$ is a solution of \eqref{eqdf3_5} 
in $A_{La,L^{3}a}(0)$ for $t$ sufficiently small. Suppose in addition 
that in the decomposition \eqref{pos_neg_log_annulus}, $h_{0}\equiv 0$. 
For any $0<\beta^{\prime}<\frac{1}{2}\beta$, there exists $L_{0}=L_{0}\left(\beta,\beta^{\prime}\right)$ with $L_{0}>1$ such that for $L>L_{0}$, if
\begin{align}
|||h|||_{aL,aL^{2}}\ge L^{\beta^{\prime}}|||h|||_{a,aL},\label{eqdf3_12}
\end{align}
then
\begin{align}
|||h|||_{aL^{2},aL^{3}}\ge L^{\beta^{\prime}}|||h|||_{aL,aL^{2}},\label{eqdf3_13}
\end{align}
and if
\begin{align}
|||h|||_{aL^{2},aL^{3}}\le L^{-\beta^{\prime}}|||h|||_{aL,aL^{2}},\label{eqdf3_14}
\end{align}
then
\begin{align}
|||h|||_{aL,aL^{2}}\le L^{-\beta^{\prime}}|||h|||_{a,La}.\label{eqdf3_15}
\end{align}
Moreover, at least one of \eqref{eqdf3_13}, 
\eqref{eqdf3_15} holds (whether or not at least one of \eqref{eqdf3_12}, 
\eqref{eqdf3_15} holds). 
\end{lemma}  

\begin{proof} By scaling properties of the norms $|||\cdot|||$, it suffices to prove the lemma for the case $a=1$. Suppose that (\ref{eqdf3_12}) holds. From the decomposition $h=h^{+}+h^{-}$ and the Cauchy-Schwarz inequality we clearly have 
\begin{align}
|||h|||^{2}_{L,L^{2}}\le 2\left(|||h^{+}|||^{2}_{L,L^{2}}
+|||h^{-}|||^{2}_{L,L^{2}}\right),
\end{align}
and then
\begin{align}\label{l2crossterms}
2\left(|||h^{+}|||^{2}_{L,L^{2}}+|||h^{-}|||^{2}_{L,L^{2}}\right)\ge
L^{2\beta^{\prime}}\left(|||h^{+}|||^{2}_{1,L}+|||h^{-}|||^{2}_{1,L}+2\langle\langle\langle h^{+},h^{-}\rangle\rangle\rangle_{1,L}\right).
\end{align}
Here $\langle\langle\langle\cdot,\cdot\rangle\rangle\rangle$ is the inner product associated to the norm $|||\cdot|||$.
From Lemma \ref{posneg}, for $L$ large enough (depending on $\beta$ and $\beta^{\prime}$) we can estimate $\langle\langle\langle h^{+},h^{-}\rangle\rangle\rangle_{1,L}$ in (\ref{l2crossterms}) as
\begin{align}\label{crossterms}
\langle\langle\langle h^{+},h^{-}\rangle\rangle\rangle_{1,L}\ge-|||h^{+}|||_{1,L}|||h^{-}|||_{1,L}\ge -L^{-\beta^{\prime}}|||h^{+}|||_{L,L^{2}}|||h^{-}|||_{1,L},
\end{align} 
and then for a fixed $0<\epsilon<\frac{1}{2}$, there exists a positive constant $c(\epsilon)$ such that (\ref{l2crossterms}) and (\ref{crossterms}) imply
 \begin{align}
L^{2\beta^{\prime}}&\left(|||h^{+}|||^{2}_{1,L}+|||h^{-}|||^{2}_{1,L}-2c(\epsilon)L^{-2\beta^{\prime}}|||h^{+}|||^{2}_{L,L^{2}}-2\epsilon|||h^{-}|||^{2}_{1,L}\right)\nonumber\\
&\le 2\left(|||h^{+}|||^{2}_{L,L^{2}}+|||h^{-}|||^{2}_{L,L^{2}}\right),
\end{align}
and hence
\begin{align}
2(1+c(\epsilon))\left(|||h^{+}|||^{2}_{L,L^{2}}+|||h^{-}|||^{2}_{L,L^{2}}\right)\ge (1-2\epsilon)L^{2\beta^{\prime}}\left(|||h^{+}|||^{2}_{1,L}+|||h^{-}|||^{2}_{1,L}\right).
\end{align}
We have shown that for fixed $0<\epsilon<\frac{1}{2}$ there exists a positive constant $q(\epsilon)$ such that
\begin{align}\label{l2nocrossterms}
\left(|||h^{+}|||^{2}_{L,L^{2}}+|||h^{-}|||^{2}_{L,L^{2}}\right)\ge q(\epsilon)L^{2\beta^{\prime}}\left(|||h^{+}|||^{2}_{1,L}+|||h^{-}|||^{2}_{1,L}\right).
\end{align} 
Combining Lemma \ref{posneg} with (\ref{l2nocrossterms}) we obtain
\begin{align}
\left(|||h^{+}|||^{2}_{L,L^{2}}+|||h^{-}|||^{2}_{L,L^{2}}\right)\ge q(\epsilon)L^{2\beta^{\prime}}|||h^{-}|||^{2}_{1,L}\ge q(\epsilon)L^{4\beta^{\prime}}|||h^{-}|||^{2}_{L,L^{2}},
\end{align}
and therefore
\begin{align}\label{l2plus>minus}
|||h^{+}|||^{2}_{L,L^{2}}\ge \left(q(\epsilon)L^{4\beta^{\prime}}-1\right)|||h^{-}|||^{2}_{L,L^{2}}.
\end{align}
On the other hand, for fixed $0<\epsilon<\frac{1}{2}$ we choose $c(\epsilon)$ as before so that 
\begin{align}
|||h|||^{2}_{L^{2},L^{3}}\ge(1-2\epsilon)|||h^{+}|||^{2}_{L^{2},L^{3}}-2c(\epsilon)|||h^{-}|||^{2}_{L^{2},L^{3}},
\end{align}
and by virtue of Lemma \ref{posneg}, for any $\beta^{\prime\prime}$ with $\beta^{\prime}<\beta^{\prime\prime}<\frac{1}{2}\beta$, there exists $L_{0}=L_{0}(\beta,\beta^{\prime\prime})>0$ such that if $L>L_{0}$ then
\begin{align}\label{l2_2R_3R}
|||h|||^{2}_{L^{2},L^{3}}
&\ge(1-2\epsilon)L^{2\beta^{\prime\prime}}|||h^{+}|||^{2}_{L,L^{2}}-2c(\epsilon)L^{-2\beta^{\prime\prime}}|||h^{-}|||_{L,L^{2}}.
\end{align}
If $L$ is large enough so that $L^{-2\beta^{\prime\prime}}<\frac{1}{2}$ we have from (\ref{l2plus>minus})
\begin{align}\label{2R3Rlowbound}
L^{2\beta^{\prime\prime}}|||h^{+}|||^{2}_{L,L^{2}}&\ge L^{2\beta^{\prime\prime}}\left(\frac{1}{2}+L^{-2\beta^{\prime\prime}}\right)|||h^{+}|||^{2}_{L,L^{2}}\ge \frac{1}{2}L^{2\beta^{\prime\prime}}|||h^{+}|||^{2}_{L,L^{2}}+|||h^{+}|||^{2}_{L,L^{2}}\nonumber \\ 
&\ge\frac{1}{2}L^{2\beta^{\prime\prime}}|||h^{+}|||^{2}_{L,L^{2}}+\left(q(\epsilon)L^{4\beta^{\prime}}-1\right)|||h^{-}|||^{2}_{L,L^{2}}.
\end{align}
Finally, from (\ref{l2_2R_3R}) and (\ref{2R3Rlowbound}) it follows that
\begin{align}
|||h|||^{2}_{L^{2},L^{3}}\ge v(\epsilon,L,\beta^{\prime},\beta^{\prime\prime})L^{2\beta^{\prime}}|||h^{+}|||^{2}_{L,L^{2}}+w(\epsilon,L,\beta^{\prime},\beta^{\prime\prime})L^{2\beta^{\prime}}|||h^{-}|||^{2}_{L,L^{2}},
\end{align}
where 
\begin{align}
v(\epsilon,L,\beta^{\prime},\beta^{\prime\prime})=\frac{1}{2}(1-2\epsilon)L^{2(\beta^{\prime\prime}-\beta^{\prime})},
\end{align}
and 
\begin{align}
w(\epsilon,L,\beta^{\prime},\beta^{\prime\prime})=(1-2\epsilon)(q(\epsilon)L^{2\beta^{\prime}}-L^{-2\beta^{\prime}})-2c(\epsilon)L^{-2(\beta^{\prime}+\beta^{\prime\prime})}.
\end{align}
It is clear that we can choose $L$ large enough so that $v(\epsilon,L,\beta^{\prime},\beta^{\prime\prime}), w(\epsilon,L,\beta^{\prime},\beta^{\prime\prime})\ge 2$ and then
\begin{align}
|||h|||^2_{L^{2},L^{3}}\ge 2L^{2\beta^{\prime}}\left(|||h^{+}|||^{2}_{L,L^{2}}+|||h^{-}|||^{2}_{L,L^{2}}\right)\ge
L^{2\beta^{\prime}}|||h|||^{2}_{L,L^{2}},
\end{align} 
as needed. The proof for the case (\ref{eqdf3_14}) is analogous. For the rest of the proposition, note that by the Cauchy-Schwarz inequality we must have either $|||h^{+}|||_{L,L^{2}}\ge \frac{1}{2}|||h|||_{L,L^{2}}$ or $|||h^{-}|||_{L,L^{2}}\ge \frac{1}{2}|||h|||_{L,L^{2}}$. If $|||h^{+}|||_{L,L^{2}}\ge \frac{1}{2}|||h|||_{L,L^{2}}$, then for fixed $0<\epsilon<1$ there exists $c_{0}(\epsilon)>1$ such that
\begin{align}
(1-c_{0}(\epsilon))|||h^{+}|||^{2}_{L,L^{2}}+(1-\epsilon)|||h^{-}|||^{2}_{L,L^{2}}\le |||h|||^{2}_{L,L^{2}},
\end{align}
and since $|||h|||^{2}_{L,L^{2}}\le 4|||h^{+}|||^{2}_{L,L^{2}}$ we conclude that for some positive constant $c_{1}(\epsilon)$ we have
\begin{align}\label{plus>minusR2R}
|||h^{+}|||^{2}_{L,L^{2}}\ge c_{1}(\epsilon)|||h^{-}|||^2_{L,L^{2}}.
\end{align}
On the other hand, if we fix $0<\epsilon<1$ there exists a constant $c_{2}(\epsilon)>0$ such that
\begin{align}
|||h|||^{2}_{L^{2},L^{3}}\ge (1-\epsilon)|||h^{+}|||^{2}_{L^{2},L^{3}}-c_{2}(\epsilon)|||h^{-}|||^{2}_{L^{2},L^{3}},
\end{align}
and from Lemma \ref{posneg}, for any $\beta^{\prime}<\beta^{\prime\prime}<\frac{1}{2}\beta$ there exists $L_{0}=L_{0}(\beta,\beta^{\prime\prime})$ such that if $L>L_{0}$ then
\begin{align}
|||h|||^2_{L^{2},L^{3}}\ge (1-\epsilon)L^{2\beta^{\prime\prime}}|||h^{+}|||^2_{L,L^{2}}-c_{2}(\epsilon)L^{-2\beta^{\prime\prime}}|||h^{-}|||^{2}_{L,L^{2}},
\end{align}
and from (\ref{plus>minusR2R}) we have
\begin{align}
|||h|||^2_{L^{2},L^{3}}\ge C(\epsilon,L,\beta^{\prime},\beta^{\prime\prime})L^{2\beta^{\prime}}|||h^+|||^{2}_{L,L^{2}},
\end{align}
where 
\begin{align}
C(\epsilon,L,\beta^{\prime},\beta^{\prime\prime})=(1-\epsilon)L^{2(\beta^{\prime\prime}-\beta^{\prime})}
-\frac{c_{2}(\epsilon)}{c_{1}(\epsilon)}L^{-2(\beta^{\prime}+\beta^{\prime\prime})}.
\end{align}
If we choose $L$ large enough so that $C(\epsilon,L,\beta^{\prime},\beta^{\prime\prime})\ge 4$ we obtain
\begin{align}
|||h|||^{2}_{L^{2},L^{3}}\ge 4L^{2\beta^{\prime}}|||h^{+}|||^{2}_{L,L^{2}}\ge L^{2\beta^{\prime}}|||h|||^{2}_{L,L^{2}},
\end{align}
as needed. In a similar way we can show that if $|||h^{-}|||^2_{L,L^{2}}\ge \frac{1}{2}|||h|||^2_{L,L^{2}}$, then we have the inequality $|||h|||^2_{L,L^{2}}\le L^{-2\beta^{\prime}}|||h|||^2_{1,L}$, which completes the proof.
\end{proof}

\subsection{Degenerate solutions of the linearized equations}\label{ws}
We now turn our attention to degenerate solutions of (\ref{eqdf3_5}).
If $t = 0$, then constants are non-trivial degenerate 
solutions, which are also divergence-free. 
However, these are not $\delta_t$-free for $t \neq 0$. 
The main result of this section is that 
there are in fact {\em{no}} degenerate solutions of (\ref{eqdf3_5}) 
for all $t$ nonzero and sufficiently small:
\begin{proposition}\label{propmod1}
There exists $t_{0}>0$ such that if $0<|t|<t_{0}$ there are no degenerate 
solutions of \eqref{eqdf3_5} subject to $\delta_{t}h=0$ on any 
annulus $A_{c,d}(0)$.  In particular, for $t\ne 0$ sufficiently 
small, Lemma {\em{\ref{lemdf2}}} holds.
\end{proposition}

\begin{proof}
We only need to consider the
case that $h$ is a finite sum in \eqref{parseval}. 
In this case, $h$ extends to a solution of $\mathcal{P}^{(k)}_{t}h=0$ on 
$\R^{n}\backslash\{0\}$ subject 
to $\delta_{t}h=0$. Let $\rho>0$, let $t_{0},t,\gamma_{0}(t)$ be as in Proposition
\ref{nolineargrowth}
 and let $\varphi$ be a $C^{\infty}$ 
function such that $\varphi|_{B_{\rho}(0)}\equiv 0$ and 
$\varphi|_{\R^{n}\backslash B_{2\rho}(0)}\equiv 1$.  Choose $p>n$
and a number $0<\gamma<\gamma_{0}(t)<1$, then 
$\varphi\delta h\in {W^{\prime}}^{0,p,0,1}_{\gamma-1}$ and 
$(1-\varphi)\delta h\in W^{0,p,0,1}_{-\gamma-1}$. Since $\gamma$ is 
nonexceptional it follows from Proposition \ref{propws1} that there 
exists $X_{1}\in {W^{\prime}}^{2,p,0,1}_{\gamma+1}$ such that 
$\Box X_{1}=\varphi\delta h$ and 
$X_{2}\in{W^{\prime}}^{2,p,0,1}_{-\gamma+1} $ such that 
$\Box X_{2}=(1-\varphi)\delta h$, therefore 
\begin{align}
\Box\left(X_{1}+X_{2}\right)=\delta h,
\end{align}
moreover, there exists $h_{0}$ such that $\delta h_{0}\equiv 0$, 
$\varphi h_{0}\in {W^{\prime}}^{1,p,0,2}_{\gamma}$,
$(1-\varphi) h_{0}\in {W^{\prime}}^{1,p,0,2}_{-\gamma}$,  and
\begin{align}
h=L_{(X_{1}+X_{2})}g_{0}+h_{0}~\mbox{on}~\R^{n}\backslash\{0\}.
\end{align}
Note that from the weighted Sobolev inequality (see \cite[Theorem 1.2]{bar}), 
we must have 
\begin{align}
h_{0}=O(r^{-\gamma}) ~\mbox{as}~ r\rightarrow 0, \ \mbox{ and }
h_{0}=O(r^{\gamma}) ~\mbox{as}~ r\rightarrow\infty\label{eqbdd2}.
\end{align}
Let $X=X_{1}+X_{2}$, from $\mathcal{P}_{t}^{(k)}h=0$ and 
$\delta_{t}h=0$, $h$ satisfies
\begin{align}\label{obtf}
\Delta^{k-1}\left(\frac{1}{2}\Delta^{2}h+\frac{1}{2}\nabla^{2} \delta \delta h +\Delta\delta^{*}\delta h\right)=0
\end{align}
on $\R^{n}\backslash\{0\}$. From diffeomorphism invariance, 
any Lie derivative $L_{X}g_{0}$ is in the kernel of the linearized operator,
and similarly, $R^{\prime}_{g_{0}}(L_{X}g_{0})=0$.
It follows that $L_{X}g_{0}$ satisfies 
\eqref{obtf} and so does $h_{0}$ on $\R^{n}\backslash\{0\}$. 
From $\delta h_{0}=0$, $h_{0}$ satisfies
\begin{align}
\Delta^{k+1}h_{0}&=0,\label{eqlapn1}
\end{align}
so we can expand $h_{0}$ in terms of homogeneous solutions 
of (\ref{eqlapn1}) on $\R^{n}\backslash\{0\}$ and by (\ref{eqbdd2}) 
and the proof of Proposition \ref{proplin1} we  
conclude that $h_{0}=\log(r) \cdot C+C^{\prime}$ where 
$C,C^{\prime}$ are matrices whose components 
are constant, but since 
$\delta h_{0}\equiv 0$ it follows from Proposition \ref{proplin1} that 
$C=0$ and $h_{0}$ is  constant 
in $\R^{n}$, in particular $h_{0}$ is a Lie derivative. We can now write  
$h=L_{Y}g_{0}$ where  $Y$ is a solution of 
$\Box_{t}Y=\delta_{t}L_{Y}g_{0}\equiv 0$ and clearly 
$Y$ is essentially linear in the sense of \eqref{lineargrowth} for some $\gamma$ 
with $0<\gamma<\gamma_{0}(t)$. But from Proposition \ref{nolineargrowth}, 
we know that for $t\ne 0$ sufficiently small, if any such solution 
$Y$ is non-zero then it must 
be dual to a Killing field which shows that $h\equiv 0$ as needed. 
In the case $n=3$, all solutions $\Delta^{2}h_{0}=0$ 
satisfying (\ref{eqbdd2}) are of the form $h_{0}=C+h_{1}$ where the components
of $C$ are constant and the components of $h_{1}$ are spherical harmonics of order 1,
however, if $\delta h_{0}\equiv 0$ then $h_{1}\equiv 0$ as seen in the proof
of Proposition \ref{proplin1}.
\end{proof}
\begin{remark} 
{\em
The argument in \cite[Corollary 3.7]{s} is 
incomplete since only radially parallel solutions are ruled 
out there. One must moreover rule out degenerate solutions 
(those with oscillatory behavior and possibly times a 
power of $\log$) which are {\em{not}} radially parallel.
}
\end{remark}
\subsection{Scaling and the nonlinear equation}\label{ene}
In this subsection we prove 
the nonlinear version of the Three Annulus Lemma. 
We assume that $(M^{n},g)$ is ALE of order~0, 
as in Definition \ref{ALEdef}.
In the following we use the ALE coordinate system 
to transfer the problem to $\left(\R^{n}\backslash B_{\rho}(0)\right)/\Gamma$. 
We have the following elliptic Schauder estimate for solutions
of~(\ref{eqsfm10}):
\begin{lemma}\label{lemdf1}
Let $0<a<d$ and let $h\in\ten^{0,2}_{m,\alpha;0}(A_{a,d})$ be a 
solution of  \eqref{eqsfm10}. 
There exists $\chi>0$ such that if $\|h\|_{\ten^{0,2}_{m,\alpha;0}(A_{a,d}(0))}<\chi$, 
then for every $b, c$ with $a<b<c<d$ one has
\begin{align}\label{scha}
\|h\|_{\ten^{0,2}_{m,\alpha;0}(A_{b,c}(0))}\le C|||h|||_{a,d},
\end{align}
with $C=C(\lambda_{1},\lambda_{2},n,m,\alpha,b-a,d-c,t)$, 
where $0 <\lambda_{1}\le\lambda_{2}$ are 
ellipticity constants of \eqref{eqsfm10}. 
\end{lemma}
\begin{proof} The result follows from standard interior elliptic 
regularity estimates, see for example \cite[Chapter II]{eid}.
Note the leading order term is a power of the Laplacian, but 
lower order coefficients are negative powers of $r$. However, 
the Schauder estimate depends only on an appropriate weighted norm of the 
coefficients, which in this case is bounded, as one can easily verify.
\end{proof}
The following scaling lemma will be used to reduce the nonlinear
problem to the linear case. 
\begin{lemma}\label{lemdf1_1} Let $\{h_{i}\}$ be a sequence of solutions 
of \eqref{eqsfm10} satisfying
\begin{align}
\|h_{i}\|_{\ten^{0,2}_{m,\alpha;0}(A_{a,L^{3}a}(0))}<\chi_{i},
\end{align}
where $\{\chi_{i}\}$ is a sequence of positive numbers such that 
$\chi_{i}\rightarrow 0$. Suppose in addition that for some 
positive constant $C$ we have 
\begin{align}
|||h_{i}|||_{a,La}+|||h_{i}|||_{L^{2}a,L^{3}a}\le C|||h_{i}|||_{La,L^{2}a}.
\end{align}
Let $q_{i}=|||h_{i}|||_{La,L^{2}a}^{-1}h_{i}$, then on any annulus 
$A_{c,d}(0)$ with $a<c<d<L^{3}a$, there exists a subsequence $q_{i_{j}}$ 
that converges in $\ten^{0,2}_{m,\alpha^{\prime};0}(A_{c,d}(0))$ with 
$\alpha^{\prime}<\alpha$ to $\tilde{q}$ satisfying \eqref{eqdf3_5}.
\end{lemma}   
\begin{proof} The sequence $\{q_{i}\}$ satisfies
\begin{align}
|||q_{i}|||_{a,La}+|||q_{i}|||_{La^{2},L^{3}a}\le C,
\end{align}
\noindent
in particular, if $c,d$ as in the statement, we have from 
Lemma \ref{lemdf1} the inequality
\begin{align}
\|q_{i}\|_{\ten^{0,2}_{m,\alpha;0}(A_{c,d}(0))}\le C^{\prime},\label{eqdf3_8}
\end{align}
for some positive constant $C^{\prime}$. By the Arzela-Ascoli theorem 
there exists a subsequence $\{q_{i_{j}}\}$ that converges in 
$\ten^{0,2}_{m,\alpha^{\prime};0}(A_{c,d}(0))$ with $\alpha^{\prime}<\alpha$ 
to $\tilde{q}$. It only remains to prove that $\tilde{q}$ solves 
(\ref{eqdf3_5}). For that purpose we write (\ref{eqsfm10}) as 
\begin{align}
\mathcal{P}_{t}^{(k)}(c_{i}q_{i})+\mathcal{R}^{(k)}(c_{i}q_{i},g_{0})=0,\label{eqdf3_9}
\end{align}
where $c_{i}=|||h_{i}|||_{La,L^{2}a}$. Note that $c_{i}\rightarrow 0$ 
as $i\rightarrow\infty$. From the estimate (\ref{eqdf3_8}) and 
equation (\ref{eqsfm7_1}) it follows that
\begin{align}
\mathcal{R}^{(k)}(c_{i}q_{i},g_{0})
=O(c_{i}^{2})~\mbox{as}~c_{i}\rightarrow 0,\label{eqee1}
\end{align}
where the bound in the right-hand side of (\ref{eqee1}) is with 
respect to the norm in $\ten^{0,2}_{m-2(k+1),\alpha^{\prime};0}(A_{c,d}(0))$,  
so (\ref{eqdf3_9}) takes the form
\begin{align}
c_{i}\mathcal{P}^{(k)}_{t}(q_{i})
=O(c^{2}_{i})~\mbox{as}~c_{i}\rightarrow 0.\label{eqee2}
\end{align}
Passing to the subsequence $\{q_{i_{j}}\}$ we conclude that the 
limit $\tilde{q}$ satisfies $\mathcal{P}^{(k)}_{t}\tilde{q}=0$. 
\end{proof}
Using Lemma \ref{lemdf2} and Proposition \ref{propmod1} we have the following nonlinear version of the Three Annulus Lemma:
\begin{lemma}\label{lemdf3}
Let $\rho,t>0$ and let $h$ be a solution of \eqref{eqsfm10} on $A_{\rho,\infty}(0)$ with $\delta_{t}h=0$. 
Let $\beta^{\prime}>0$ and $L_{0}>1$ be as in Lemma \ref{lemdf2} 
and let $L,a>0$ be such that $L_{0}a>\rho$ and $L>L_{0}$ .  There exist 
$\chi=\chi(n,\lambda,\Lambda)>0$ so that if $|h|_{\ten^{0,2}_{m,\alpha;0}(A_{\rho,\infty}(0))}<\chi$, then if
\begin{align}
|||h|||_{La,L^{2}a}\ge L^{\beta^{\prime}}|||h|||_{a,La},\label{eqdf1}
\end{align}
then
\begin{align}
|||h|||_{L^{2}a,L^{3}a}\ge L^{\beta^{\prime}}|||h|||_{La,L^{2}a},\label{eqdf2}
\end{align}
and if
\begin{align}
|||h|||_{L^{2}a,L^{3}a}\le L^{-\beta^{\prime}}|||h|||_{La,L^{2}a},\label{eqdf3}
\end{align}
then
\begin{align}
|||h|||_{La,L^{2}a}\le L^{-\beta^{\prime}}|||h|||_{a,La}.\label{eqdf4}
\end{align}
Moreover, there exists $t_{0}>0$ such that if $0<|t|<t_{0}$ then at least one of 
\eqref{eqdf2},\eqref{eqdf4} must hold.
\end{lemma}
\begin{proof} If any of the implications in the statement of the lemma fails we can use the rescaling construction in Lemma \ref{lemdf1_1} to  
produce a solution of the linearized equation that contradicts Lemma \ref{lemdf2}. 
\end{proof}
\subsection{Global divergence-free gauges}\label{gdf}
From (\ref{eqgdfale1in}) and (\ref{eqgdfale2in}) it follows that $\Psi_{*}g-g_{0}$ 
satisfies the estimate
\begin{align}
|\left(\Psi_{*}g-g_{0}\right)_{(r,x)}|_{m,\alpha;0}
=o(1)~\mbox{as}~r\rightarrow\infty.\label{eqgdf1}
\end{align}
for all $m\ge 0$. 
\begin{remark}\label{remale1} {\em
We do not need to assume decay on derivatives of arbitrary order in
\eqref{eqgdfale2in}. Our proof only requires 
that the estimate (\ref{eqgdfale2in}) be satisfied only for all 
multi-indices $l$ such that $1\le |l|\le 2(k+1)+1$. 
Then (\ref{eqgdf1}) holds for all $m\le 2(k+1)$. }
\end{remark}
\noindent
With the decay in (\ref{eqgdf1}) we can prove the existence of 
$\delta_{t}$-free gauges on certain annuli by means of the implicit 
function theorem. The following result is contained in \cite[Theorem 3.1]{ct} 
and does \emph{not} depend on the metric $g$ being $\Omega^{(k)}$-flat:

\begin{proposition}\label{propdf0} Let $t_{0}$ be as in Proposition 
{\em{\ref{nolineargrowth}}} 
and let $t$ be such that $0<|t|<t_{0}$. There exists $\chi=\chi(t,m)$ such that if 
$(C(N^{n-1}),g_{0})$ is a Ricci-flat cone and $\tilde{g}$ is a 
metric on $A_{c,\infty}(0)\subset C(N^{n-1})$  such that
\begin{align}
\|\tilde{g}-g_{0}\|_{\ten^{0,2}_{m,\alpha;0}(A_{c,\infty}(0))}<\chi,\label{eqdf0_1}
\end{align}
then there exists a diffeomorphism $\phi=\phi(\tilde{g})$, $\phi:A_{c,\infty}(0)\rightarrow A_{c,\infty}(0)$ such that
\begin{align}
\phi^{*}\tilde{g}\in\ten^{0,2}_{m,\alpha;0}(A_{c,\infty}(0)),\label{eqdf0_2}
\end{align}
and 
\begin{align}
\delta_{t}(\phi^{*}\tilde{g}-g_{0})=0.\label{eqdf0_3}
\end{align}
Moreover, if 
\begin{align}
\|\delta_{t}(\tilde{g})\|_{\ten^{0,1}_{m-1,\alpha;-1}(A_{c,\infty}(p))}<\epsilon,
\label{eqdf0_4}
\end{align}
then
\begin{align}
\|\phi^{*}\tilde{g}-\tilde{g}\|_{\ten^{0,2}_{m,\alpha;0}(A_{c,\infty}(p))}
<\delta(\epsilon),\label{eqdf0_5}
\end{align}
where $\delta(\epsilon)\searrow 0$ as $\epsilon\rightarrow 0$. 
\end{proposition} 
\noindent
In our application of this Lemma,  we will simply let $\tilde{g} = \Psi_* g$, and
$g_0$ the flat metric on the Euclidean cone (recall Definition \ref{ALEdef}).  
We will then write $h = \phi^* \tilde{g} - g_0$. 
From Proposition \ref{propdf0} and using that $g$ is ALE of order 0 we have 
\begin{lemma}\label{cladf1} 
Let $h$ and $A_{c,\infty}(0)$ be as above, 
and let $L_{0}>0$ be as in Lemma \ref{lemdf2}. For any $L>L_{0}$ and any $a>c$ we have $\lim_{i\rightarrow\infty}|||h|||_{L^{i}a,L^{i+1}a}=0$, 
moreover, we have the inequality
\begin{align}
|||h|||_{L^{i}a,L^{i+1}a}\le L^{-\beta^{\prime}(i-i_{0})}|||h|||_{L^{i_{0}}a,L^{i_{0}+1}a}~\mbox{for all}~i\ge i_{0}.\label{eqgdf2}
\end{align}
\end{lemma}   

\begin{proof} By Lemma \ref{lemdf3}, for any $i> i_{0}$ we have either
\begin{align}
|||h|||_{L^{i+1},L^{i+2}}\ge L^{\beta^{\prime}}|||h|||_{L^{i},L^{i+1}}\label{anngrowth},
\end{align} 
or
\begin{align}
|||h|||_{L^{i},L^{i+1}}\le L^{-\beta^{\prime}}|||h|||_{L^{i-1},L^{i}}\label{anndecay},
\end{align}
Suppose that we have (\ref{anngrowth}), then we conclude again from Lemma \ref{lemdf3} that we must have for any integer $s\ge 1$ the inequality
\begin{align}
|||h|||_{L^{i+s},L^{i+s+1}}\ge L^{\beta^{\prime}s}|||h|||_{L^{i},L^{i+1}}\label{powergrowth}.
\end{align}
Since $g$ is ALE of order 0, we can use (\ref{eqdf0_4}) in Proposition \ref{propdf0} 
to conclude that for $\epsilon>0$ sufficiently small and $s>1$ large we must have 
\begin{align}
\|\Psi_{*}g-g_{0}\|_{\ten^{0,2}_{m,\alpha;0}(A_{L^{i+s},L^{i+s+1}}(0))}&<\epsilon,\\
\|\phi^{*}\Psi_{*}g-\Psi_{*}g\|_{\ten^{0,2}_{m,\alpha;0}(A_{L^{i+s},L^{i+s+1}(0)})}&<\frac{1}{2},
\end{align}  
and therefore 
\begin{align}
\|h\|_{\ten^{0,2}_{m,\alpha;0}(A_{L^{i+s},L^{i+s+1}}(0))}<1.
\end{align}
In particular, we must have
\begin{align}
|||h|||_{L^{i+s},L^{i+s+1}}\le c_{n}\log(L),\label{logbound}
\end{align}
where $c_{n}$ is a dimensional constant. It is clear that (\ref{powergrowth}) contradicts (\ref{logbound}) for $s$ large. It follows that \eqref{anndecay} 
holds and then, by Lemma \ref{lemdf3} we must have (\ref{eqgdf2}).  
\end{proof}

\noindent
From Lemma \ref{cladf1} we have the following improvement in the ALE order of $g$:

\begin{corollary}\label{corgdf1}
If $g$ is 
ALE of order 0, scalar flat and (extended) obstruction-flat or satisfies 
\eqref{eqsrt11in},
then there exists an annulus of the form $A_{c^{\prime},\infty}(0)$ and a 
diffeomorphism 
$\phi:A_{c^{\prime},\infty}(0)\rightarrow A_{c^{\prime},\infty}(0)$ 
such that $|\phi^{*}\Psi_{*}g-g_{0}|_{m,\alpha;0}=O(r^{-\beta^{\prime}})$ as $r\rightarrow\infty$ and therefore, $g$ is ALE of order $\beta^{\prime}$.
\end{corollary}
\begin{proof}
From  Lemma \ref{lemdf1}, Proposition \ref{propdf0},
and Lemma \ref{cladf1}, 
there exists a constant $C>0$ such 
that for any $(r,x)\in A_{c^{\prime},\infty}(0)$ with $r$ sufficiently large we have
\begin{align*}
|(\phi^{*}\Psi_{*}g-g_{0})_{(r,x)}|_{m,\alpha;0}\le C r^{-\beta^{\prime}},
\end{align*}
so the claim follows.  
\end{proof}

From \cite[Sections 2 and 3]{ct} we have 

\begin{proposition}\label{dftau}
Suppose $g$ is a metric defined on $\R^{n}\backslash B_{\rho}(0)$ and satisfies
\begin{align}
g-g_{0}=O(r^{-\tau}) ~\text{as}~r\rightarrow\infty,
\end{align}
then there exists a diffeomorphism $\phi:\R^{n}\backslash B_{\rho}(0)\rightarrow\R^{n}\backslash B_{\rho}(0)$ such that $h=\phi^{*}g-g_{0}$ satisfies $h\in\ten^{0,2}_{m,\alpha;-\tau}(\R^{n}\backslash B_{\rho}(0))$ and $\delta_{g_0} h=0$.
\end{proposition}
We summarize the results of this section in the following corollary
\begin{corollary}\label{remgdf2} 
If $g$ is 
ALE of order 0, scalar flat and (extended) obstruction-flat or satisfies 
\eqref{eqsrt11in}, then there 
exists a diffeomorphism 
$\Phi:M\backslash K\rightarrow (\R^{n}\backslash B_{\rho}(0))\slash \Gamma$ 
for some $\rho>0$ and a compact set $K\subset M$ such that 
$h=\Phi_{*}g-g_{0}$ satisfies $h\in\ten^{0,2}_{m,\alpha;-\beta^{\prime}}$
and $\delta_{g_{0}}h=0$.
\end{corollary}

\section{Optimal ALE order} 
\label{osec}
In this Section, we complete the proof of Theorems \ref{maint} and \ref{maint2}. 
\subsection{Weighted Sobolev spaces and $\Delta^{k+1}$}\label{ev}
In this section we state some properties of the weighted Sobolev spaces 
introduced in Section \ref{ws} that will be useful to improve the decay 
estimate for the metric $g$ derived in Section \ref{gdf}. Throughout this 
section we will work only with $(0,2)$ tensors so when we write 
${W^{\prime}}^{m,p}_{\delta}$ we actually mean the space 
${W^{\prime}}^{m,p,0,2}_{\delta}$. We start by defining the set 
of exceptional values for $\Delta^{k+1}$.

\noindent
\begin{definition}{\em
A number $\delta\in\R$ is said to be \emph{exceptional} for $\Delta^{k+1}$ 
if $\delta$ is in the set 
\begin{align}
E=\left\{
\begin{array}{ll}
\{j\in\mathbb{Z}:j\ne -1,-2,\ldots,2(k+1)-(n-1)\}&\mbox{if}~ n>2(k+1)\\
\mathbb{Z}&\mbox{if}~ n=2(k+1). 
\end{array}
\right.
\end{align}
We say that $\delta$ is {\em{nonexceptional}} if $\delta\in\R\backslash E$.}
\end{definition}

\begin{remark}{\em
The exceptional values for $\Delta^{k+1}$ correspond to the growth rates 
of solutions of  $\Delta^{k+1}h=0$ on the complement of a ball, however, 
when $n=2(k+1)$, as observed in the proof of Proposition \ref{proplin1},
there are solutions of $\Delta^{k+1}h=0$ on $\R^{n}\backslash\{0\}$ 
that are $O(\log(r))$ as $r\rightarrow\infty$. }
\end{remark}

\begin{lemma}\label{lemev1} If $\delta$ is nonexceptional, the map 
$\Delta^{k+1}:{W^{\prime}}^{2(k+1),p}_{\delta}
\rightarrow {W^{\prime}}^{0,p}_{\delta-2(k+1)}$ 
is an isomorphism.
\end{lemma}

\begin{proof} See \cite[Theorem 1.7]{bar}.
\end{proof}

\begin{lemma}\label{lemev2} Suppose that $h$ is defined on 
$\R^{n}\backslash B_{\rho}(0)$ and satisfies $h=O(r^{\delta})$ as 
$r\rightarrow\infty$ and assume that 
$\Delta^{k+1}h=O(r^{\delta^{\prime}-2(k+1)})$ as $r\rightarrow\infty$ 
with $\delta^{\prime}<\delta$. 
Then, for any $\tau>0$ such that $\delta^{\prime}+\tau$ is nonexceptional 
there exists $h^{\prime}\in {W^{\prime}}^{2(k+1),p}_{\delta^{\prime}+\tau}$ 
and a ball $B_{\rho^{\prime}}(0)$ such that
\begin{align}
\Delta^{k+1}(h-h^{\prime})
=0~~\mbox{on}~\R^{n}\backslash B_{\rho^{\prime}}(0),\label{eqev1}
\end{align}
Furthermore, if $2(k+1)<n$, there exists an exceptional value 
$j\le\max\{\delta,\delta^{\prime}+\tau\}$ such that
\begin{align}
h-h^{\prime}=p_{j}+O(r^{j-1})~\mbox{as}~r\rightarrow\infty,\label{eqev2}
\end{align}
where $p_{j}$ is homogeneous of degree $j$ and satisfies 
$\Delta^{(k+1)}(p_{j})=0$ on $\R^{n}\backslash\{0\}$. If $n=2(k+1)$, 
we may also have
\begin{align}
h-h^{\prime}=A \cdot \log(r)+O(1)~\mbox{as}~r\rightarrow\infty,
\end{align}
where the components of $A$ are constant.

\end{lemma}

\begin{proof} Let $\varphi$ be a function in $C^{\infty}(\R^{n})$ such that 
$\varphi\equiv 0$ on $B_{\rho}(0)$ and $\varphi\equiv 1$ on 
$\R^{n}\backslash B_{2\rho}(0)$, then 
$\Delta^{k+1}\left(\varphi h\right)\in {W^{\prime}}^{0,p}_{\delta^{\prime}-2(k+1)+\tau}$ 
for every $\tau>0$. If we choose $\tau$ in such a way that 
$\delta^{\prime}+\tau$ is nonexceptional, by Lemma \ref{lemev1} there exists 
$h^{\prime}\in {W^{\prime}}^{2(k+1)}_{\delta^{\prime}+\tau}$ such that   
\begin{align}
\Delta^{(k+1)}\left(h\varphi-h^{\prime}\right)=0,
\end{align}
and on $\R^{n}\backslash B_{2\rho}(0)$ we have (\ref{eqev1}). 
The expansion (\ref{eqev2}) follows from the expansion at infinity of 
solutions of $\Delta^{k+1}\tilde{h}=0$ on $\R^{n}\backslash B_{1}(0)$ 
(compare with the proof of Proposition~\ref{propmod1}).   
\end{proof}

\subsection{Optimal decay}\label{od}
Suppose that $(M^{n},g)$ is ALE of order 0, scalar-flat, 
and either $\Omega^{(k)}$-flat or satisfies \eqref{eqsrt11in}. 
Corollary \ref{remgdf2} showed that $g$ is ALE of order $\beta'$
for some $\beta' > 0$. 
In the next proposition we obtain the optimal value for $\beta^{\prime}$
as stated in Theorems \ref{maint} and \ref{maint2}. 
\begin{proposition}\label{propopt1}
Let $h$ be as in Corollary {\em{\ref{remgdf2}}}.
If $g$ is $\Omega^{(k)}$-flat or satisfies \eqref{eqsrt11in}, then 
$h\in\ten^{0,2}_{m,\alpha;2k-n}$ and $g$ is ALE of order $n-2k$.  
\end{proposition}

\begin{proof}
Let us treat first the case $(k+1)=\frac{n}{2}$. Since $\delta h=0$, 
$h$ satisfies (\ref{eqsfm10}) with $t=0$, i.e.  
\begin{align}
 \frac{c_{n,\frac{n}{2}-1}}{2(n-2)} \Delta^{\frac{n}{2}}h
=\mathcal{R}^{(\frac{n}{2}-1)}(h,g_{0}),\label{eqod7}
\end{align}
where $\mathcal{R}^{(\frac{n}{2}-1)}(h,g_{0})$ is given by equation (\ref{eqsfm7_1}). 
Since $h = O(r^{- \beta'})$ as $r \rightarrow \infty$, 
the least decaying terms in (\ref{eqsfm7_1}) 
are those terms containing $\nabla^{\alpha_{1}}h*\nabla^{\alpha_{2}}h$ with 
$\alpha_{1}+\alpha_{2}=n$, so from (\ref{eqod7}) we obtain 
$\Delta^{\frac{n}{2}}h=O(r^{-2\beta^{\prime}-n})$, and by Lemma (\ref{lemev2}), 
for any $\tau>0$ such that $-2\beta^{\prime}+\tau$ is nonexceptional there 
exists $h_{1}\in {W^{\prime}}^{n,p}_{-2\beta^{\prime}+\tau}$ such that 
$\Delta^{\frac{n}{2}}(h-h_{1})=0$ on the complement of some ball. From 
the weighted Sobolev inequality, if we take $p>n$ then $h_{1}=O(r^{-2\beta^{\prime}+\tau})$ 
as $r\rightarrow\infty$ and clearly we can assume that 
$-2\beta^{\prime}+\tau<-\beta^{\prime}$, so that both $h,h_{1}$ have pointwise 
decay at infinity but $h_{1}$ has a better decay at infinity than $h$. 
By Lemma \ref{lemev2}, and since $-1$ is the least negative exceptional 
value for $\Delta^{\frac{n}{2}}$, the difference $h-h_{1}$ has an expansion 
at infinity of the form
\begin{align}
h-h_{1}=F_{1}+O(r^{-2}),\label{eqod8}
\end{align}
where $F_{1}$ is a homogeneous solution of degree $-1$ of $\Delta^{\frac{n}{2}}h=0$ 
on $\R^{n}\backslash\{0\}$. From the proof of Proposition \ref{proplin1}, 
any such $F_{1}$ has the form 
\begin{align}
(F_{1})_{ij}(x)=u_{ij}\left(\frac{x}{|x|^{2}}\right),\label{eqod9}
\end{align}
where $u_{ij}$ are linear functions. We now claim that on the complement of 
some ball, $h$ satisfies
\begin{align}
h=F_{1}+O(r^{-1-\epsilon}),\label{eqod10}
\end{align}
for some $\epsilon>0$. If $-2\beta^{\prime}<-1$ we can choose $\tau>0$ 
sufficiently small so that $h$ satisfies (\ref{eqod10}). If not, we have 
$h=O(r^{-2\beta^{\prime}+\tau})$ with $-2\beta^{\prime}+\tau<-\beta^{\prime}$ 
and again from (\ref{eqod7}) it follows that 
$\Delta^{\frac{n}{2}}h=O(r^{-4\beta^{\prime}+2\tau})$ and we can argue as 
above to obtain $h=O(r^{-\min\{1,4\beta^{\prime}-2\tau^{\prime}\}})$ for some 
$\tau^{\prime}>0$ such that $-4\beta^{\prime}+\tau^{\prime}<-2\beta^{\prime}+\tau$. 
It is clear that we can use induction to obtain (\ref{eqod10}). Note that if 
$\delta F_{1}$ is not identically zero then $\delta F_{1}=O(r^{-2})$ as 
$r\rightarrow\infty$ but we do \emph{not} have $\delta F_{1}=O(r^{-2-\epsilon})$  
as $r\rightarrow\infty$ for any $\epsilon>0$, however, by (\ref{eqod10}) one has 
\begin{align}
\delta h=\delta F_{1}+O(r^{-2-\epsilon})~\mbox{as}~r\rightarrow\infty,
\end{align}
therefore, from $\delta h\equiv 0$ it follows that $\delta F_{1}\equiv 0$, 
and by Proposition \ref{proplin1}, $F_{1}\equiv 0$ which shows that 
$h=O(r^{-\gamma})$ with $\gamma>1$. By  (\ref{eqod7}) one obtains 
$\Delta^{\frac{n}{2}}h=O(r^{-2\gamma-n})$ and repeating the argument used 
to obtain (\ref{eqod10}) we can show that on the complement of some ball, 
$h$ has an expansion of the form
\begin{align}
h=F_{2}+O(r^{-2-\epsilon})~\mbox{as}~r\rightarrow\infty,\label{eqod11}
\end{align}
where $F_{2}$ is homogeneous of degree $-2$ and satisfies 
$\Delta^{\frac{n}{2}}F_{2}\equiv 0$ on $\R^{n}\backslash\{0\}$ and $\epsilon$ 
is some positive number. This proves that $h=O(r^{-2})$ as $r\rightarrow\infty$ 
as needed. For the case $2(k+1)<n$, the only difference with the previous proof 
is that we have to consider homogeneous solutions of $\Delta^{k+1}h\equiv 0$ on 
$\R^{n}\backslash\{0\}$ that decay like $r^{2(k+1)-n}$ and like $r^{2k+1-n}$,
but as shown in the proof of Proposition \ref{proplin1}, these solutions are 
\emph{not} divergence-free unless they are identically zero.       
\end{proof}

\section{Singularity removal theorems}\label{srt}

In this section we present the proofs of Theorems \ref{tr1} and \ref{tr2}.
\begin{lemma}\label{propsrt1}
Let $g=g_{0}+h$ be a metric defined on $B_{\rho}(0)\backslash\{0\}$ with constant 
scalar curvature, and assume that $g$ is either $\Omega^{(k)}$-flat 
or satisfies $\eqref{eqsrt11inorb}$. Suppose in addition that $\delta_{t}h=0$ 
on $B_{\rho}(0)\backslash\{0\}$. Then, on $B_{\rho}(0)\backslash\{0\}$, $h$ 
satisfies the equation
\begin{align}
\mathcal{P}^{(k)}_{t}h+\mathcal{R}^{(k)}(h,g_{0}) =0,
\label{eqsrt1}
\end{align}
where $\mathcal{P}^{(k)}_{t}$ and $\mathcal{R}^{(k)}(h,g_{0})$ have the same 
expressions as in \eqref{eqsfm6} and \eqref{eqsfm7_1} respectively. 
The operator $\mathcal{P}^{(k)}_{t}$ is elliptic.
\end{lemma}

\begin{proof}
Suppose that $R(g_{0}+h)=c$ where $c$ is a constant. If we also have 
$\delta_{t}h=0$ on $B_{\rho}(0)\backslash\{0\}$, then from $R(g_{0}+h)-R(g_{0})=c$ 
we conclude that $h$ satisfies the equation
\begin{align}
\Delta\tr(h)=-c+t\delta i_{r^{-1}\frac{\partial}{\partial r}}h+F^{\prime}(h,g_{0}),
\label{eqsrt2}
\end{align}
on $B_{\rho}(0)\backslash\{0\}$. We now write the equation 
$\Omega^{(k)}(g_{0}+h)- \Omega^{(k)}(g_{0})=0$ as
\begin{align}
0=(\Omega^{(k)})^{\prime}_{g_0}(h)+F^{(k)}(h,g_{0}),\label{eqsrt3}
\end{align}
with $F^{(k)}(h,g_{0})$ as in (\ref{eqrem10}). Using (\ref{eqlin1}), 
we see that if we insert (\ref{eqsrt2}) into (\ref{eqsrt3}) then $h$ 
satisfies (\ref{eqsrt1}) on $B_{\rho}(0)\backslash\{0\}$. 
The rest of the claim follows easily, and the same argument 
works for $\eqref{eqsrt11inorb}$.
\end{proof}
Recalling the $C^{0}$-orbifold condition as defined in 
Definition \ref{c0def}, we now assume that there exists 
a coordinate system around the origin such that
\begin{align}
g_{ij}&=\delta_{ij}+o(1),\label{eqorb1}\\
\partial^{l}g_{ij}&=o(r^{-|l|}),\label{eqorb2}
\end{align}
for any multi-index $l$ with $|l|\ge 1$ as $r \rightarrow 0$. 
\begin{remark}
\label{ordrm2}
{\em
As in the ALE case (compare Remark \ref{remale1}), we do not need an 
assumption on derivatives of arbitrary order.  
If $l$ in (\ref{eqorb2}) only satisfies $|l|\le 2(k+1)+1$, 
then we have 
$|\left(g-g_{0}\right)_{(r,x)}|_{m,\alpha;0}=o(1)$ as $r\rightarrow 0$ 
for any $m\le 2(k+1)$, which is sufficient for our proof.} 
\end{remark}
Next, we state the existence of a divergence-free gauge.
\begin{lemma}\label{lemsrtdf} Suppose that $g$ defined on $B_{\rho}(0)$ has 
constant scalar curvature and is (extended) obstruction-flat or
satisfies $\eqref{eqsrt11inorb}$. 
Suppose also that the origin is a $C^{0}$-orbifold 
point for $g$. 
Then for some $\rho^{\prime}<\rho$ there exists a diffeomorphism 
$\phi:B_{\rho^{\prime}}(0)\backslash\{0\}\rightarrow B_{\rho^{\prime}}(0)\backslash\{0\}$ such that $\delta_{g_{0}}\phi_{*}g=0$ on $B_{\rho^{\prime}}(0)\backslash\{0\}$. 
Moreover, there exists $\sigma>0$  such that $|\phi_{*}g - g_0|=O(r^{\sigma})$ as 
$r\rightarrow 0$ and $\partial^{l}\phi_{*}g=O(r^{\sigma-|l|})$ for any multi-index 
$l$ with $|l|\ge 1$.
\end{lemma}

\begin{proof} This follows from a straightforward modification 
of the proof of Corollary~\ref{remgdf2}. 
\end{proof}

\noindent
We will also need the following

\begin{lemma}\label{lemlie} If the components of $h\in S^{2}(T^{*}\R^{n})$ are linear functions then $h=L_{X}g_{0}$ for some quadratic vector field. 
\end{lemma}

\begin{proof}
Let $S^{2}_{1}$ be the subspace of $S^{2}(T^{*}\R^{n})$ consisting of all elements whose components are linear functions. If $h\in S^{2}_{1}$,  we can write the components of $h$ as
\begin{align}
h_{ij}(x)=\sum_{l=1}^{n}A_{ijl}x_{l},
\end{align}
where $A_{ijk}$ is symmetric in $i,j$ and therefore $\dim(S^{2}_{1})=\frac{n^{2}(n-1)}{2}$. On the other hand, if we let $\Gamma^{1}_{2}$ be the space of all vector fields whose components are functions which are homogeneous of degree 2, then any $X\in\Gamma^{1}_{2}$ can be written as $X=X_{i}\frac{\partial}{\partial x^{i}}$ where
\begin{align}
X_{i}(\xi)=\sum_{l,m}a_{ilm}\xi_{l}\xi_{m},
\end{align}  
with $a_{ilm}$ symmetric in $l,m$, and therefore $\dim\left(\Gamma^{1}_{2}\right)=\frac{n^{2}(n-1)}{2}=\dim(S^{2}_{1})$. Since there are no quadratic Killing vector fields, the map $\mathcal{L}:\Gamma^{1}_{2}\rightarrow S^{2}_{1}$ defined by $\mathcal{L}(X)=L_{X}g_{0}$ is an isomorphism.
\end{proof}

\begin{lemma}\label{lemode}
Let $X$ a vector field that is homogeneous of degree 2 and let $K_{X}$ be the 
diffeomorphism generated by taking the flow of $X$ to time 1 (which exists for 
$r$ sufficiently small). If $g_{0}$ 
is the Euclidean metric we have
\begin{align}
K_{X}^{*}g_{0}-L_{X}g_{0}-g_{0}=O(r^{2})~\mbox{as}~r\rightarrow 0.
\end{align}
\end{lemma}

\begin{proof} Let $\phi_{t}$ be the flow of $X$, then we have for any $t>0$
\begin{align}
\phi_{t}^{*}g_{0}=g_{0}+tL_{X}g_{0}+E(t),
\end{align}
where $E(t)$ is an error term that can be estimated as
\begin{align}
|E(t)|\le\frac{t^{2}}{2}\sup_{s\in[0,t]}\left(\left|\frac{\partial^{2}}{\partial s^{2}}
\phi_{s}\right|\right),\label{eqvf1}
\end{align}
here $|\cdot|$ is the usual pointwise norm on $S^{2}(T^{*}\R^{n})$. In particular 
\begin{align}
K_{X}^{*}g_{0}=g_{0}+L_{X}g_{0}+E(1).\label{eqvf2}
\end{align}
Since $X$ is homogeneous of degree 2, we have for any $p\in\R^{n}$
\begin{align}
\frac{\partial}{\partial t}|\phi_{t}(p)|^{2}
=2\langle\frac{\partial}{\partial t}\phi_{t}(p),
\phi_{t}(p)\rangle\le C|\phi_{t}(p)|^{3}\label{eqode1},
\end{align}
for some constant $C>0$ that only depends on $X$. Letting $r$ denote the 
distance of $p$ to the origin, it follows from (\ref{eqode1}) that for 
$0<r<\frac{1}{C}$ and $0\le t\le 1$ we have the inequality
\begin{align}
|\phi_{t}(p)|\le \frac{2r}{2-Crt},\label{eqode2}
\end{align}
and then $|\phi_{t}(p)|=O(r)$ as $r\rightarrow 0$. A similar argument shows that for 
all first order partial derivatives one has 
\begin{align}
|\partial_{l}\phi_{t}|=O(1) ~\mbox{as}~ r\rightarrow 0.\label{eqode3}
\end{align}
Since $g_{0}$ is the Euclidean metric, we have 
\begin{align*}
(\phi^{*}_{t}g_{0})_{ij}=\sum_{k,j}\partial_{k}(\phi_{t})_{i}\partial_{l}(\phi_{t})_{j},
\end{align*}
and using the chain rule we can write schematically 
\begin{align}
\frac{\partial^{2}}{\partial t^{2}}\phi_{t}^{*}g_{0}
=(\partial^{2}X)(\phi_{t})*X(\phi_{t})*\partial\phi_{t}*\partial\phi_{t}
+(\partial X)(\phi_{t})*(\partial X)(\phi_{t})*\partial\phi_{t}*\partial\phi_{t},
\label{eqode4}
\end{align}
and by (\ref{eqode3}),(\ref{eqode4}) we conclude that for $r$ sufficiently small 
\begin{align}
\left|\frac{\partial^{2}}{\partial t^{2}}\phi_{t}^{*}g_{0}\right|\le C^{\prime}r^{2},
\label{eqode5}
\end{align}
for $C^{\prime}$ depending only on $X$. By  (\ref{eqvf1}),(\ref{eqvf2}) and 
(\ref{eqode5}) the result follows.
\end{proof}

\begin{lemma}\label{lemsrt1} Let g a metric defined on $B_{\rho}(0)$ with a 
$C^{0}$-orbifold point at the origin. Suppose that $g$ has constant scalar 
curvature and is (extended) obstruction-flat or satisfies $\eqref{eqsrt11inorb}$ on 
$B_{\rho}(0)\backslash\{0\}$. Then there exists a change of 
coordinates $\tilde{\phi}$, defined in some small neighborhood around the origin, 
such that $\tilde{\phi}_{*}g$ satisfies 
\begin{align}
g&=g_{0}+O(|x|^{2}),\label{eqsrt8}\\
\partial^{l}g&=O(|x|^{2-|l|}),\label{eqsrt9}
\end{align}
for any multi-index $l$ with $|l|\ge 1$ as $|x|\rightarrow 0$. 
\end{lemma}

\begin{proof}
Let $\phi$ be the gauge given by Lemma \ref{lemsrtdf} and if we take 
$h=\phi_{*}g-g_{0}$ we obtain
\begin{align}
\frac{c_{n,k}}{2(n-2)}\Delta^{k+1}h=\mathcal{R}^{(k)}(h,g_{0})\label{eqell1}.
\end{align}
Since $h=O(r^{\sigma})$ as $r\rightarrow 0$ we have 
$\Delta^{k+1}h=O(r^{2\sigma-2(k+1)})$ as $r\rightarrow 0$ and 
as in the proof of Proposition \ref{propopt1}, for $p>n$  and for $\tau>0$ 
such that $2\sigma-\tau$ is nonexceptional and positive, there exists 
$h^{\prime}\in {W^{\prime}}^{n,p}_{2\sigma-\tau}$ such that 
$\Delta^{k+1}\left(h-h^{\prime}\right)=0$ on 
$B_{\rho}(0)\backslash\{0\}$. Since both $h,h^{\prime}$  
are $o(1)$ as $r\rightarrow 0$ we conclude that
\begin{align}
h-h^{\prime}=G_{1}+O(r^{2})~\mbox{as}~r\rightarrow 0,
\end{align}
and the components of $G_{1}$ are linear functions. As in Proposition 
\ref{propopt1} we can use induction to show that $h$ satisfies
\begin{align}
h=G_{1}+O(r^{1+\epsilon})~\mbox{as}~r\rightarrow 0,
\end{align}
for some $\epsilon>0$. The strategy for proving (\ref{eqsrt8}), (\ref{eqsrt9}) 
is slightly different to that used to prove Proposition \ref{propopt1}, but 
is still based on an argument used in \cite{ct}. From $\delta h=0$ on 
$B_{\rho^{\prime}}(0)\backslash\{0\}$, it follows that $\delta G_{1}\equiv 0$ 
and by Lemma \ref{lemlie}, $G_{1}=L_{X}g_{0}$ for some vector field $X$ such 
that $X(p)$ is homogeneous of degree $2$ in $p$. Assume that 
$\rho^{\prime}$ is sufficiently small so that $K_{X}$, the 
diffeomorphism obtained by taking the flow of $X$ to time 1, is defined. 
By Lemma \ref{lemode}
\begin{align}
K_{X}^{*}g_{0}-g_{0}-L_{X}g_{0}=O(r^{2})~\mbox{as}~r\rightarrow 0,
\end{align}
and from
\begin{align}
\left(\phi_{*}g-K_{X}^{*}g_{0}\right)+\left(K_{X}^{*}g_{0}-g_{0}-L_{X}g_{0}\right)
=O(r^{\min\{2,1+\epsilon\}})~\mbox{as}~r\rightarrow 0,
\end{align}
we conclude that
\begin{align}
K_{-X}^{*}\phi_{*}g-g_{0}=O(r^{\min\{1+{\epsilon,2}\}})~\mbox{as}~r\rightarrow 0.
\end{align}
As in Corollary \ref{remgdf2}, 
we can find a diffeomorphism $\phi^{\prime}$ 
defined on a smaller ball such that
\begin{align}
h'=\phi^{\prime}_{*}K_{-X}^{*}\phi_{*}g-g_{0},
\end{align}
\noindent
is divergence-free and $h'=O(r^{\min\{1+\epsilon,2\}})$ as $r\rightarrow 0$. 
With this new $h'$ we argue again as in 
the proof of Proposition \ref{propopt1} to obtain (\ref{eqsrt8}) and (\ref{eqsrt9}) 
as needed. 
\end{proof}
\begin{lemma}\label{lemeuc}
In the coordinate system constructed in Lemma {\em{\ref{lemsrt1}}} we have
\begin{align}
\nabla^{l}Rm=\partial^{l}Rm+O(r^{1-|l|})~\mbox{as}~r\rightarrow 0,\label{eqeuc1} 
\end{align}
for any multi-index $l$ with $|l|\ge 1$ and
\begin{align}
\Delta_{g}^{m}Ric(g)=\Delta_{g_{0}}^{m}Ric(g)
+O(r^{-2(m-1)})~\mbox{as}~r\rightarrow 0.\label{eqeuc2}
\end{align}
with $m\ge 1.$
\end{lemma}
\begin{proof} 
To show (\ref{eqeuc1}), we consider first the case $|l|=1$ and we 
write $\nabla Rm$ schematically as
\begin{align}
\nabla Rm=\partial Rm+\Gamma*Rm,
\end{align}
and note that the terms $\Gamma*Rm$ are $O(r)$ as $r\rightarrow 0$. 
The general case follows easily by induction. For (\ref{eqeuc2}) we 
start also with the case $m=1$ and write 
\begin{align}
\begin{split}\label{eqeuc3}
\Delta_{g}Ric &=g^{-1}*\nabla\nabla Ric 
=g^{-1}*\nabla\left(\partial Ric+\Gamma*Ric\right)\\
& =g^{-1}*\left(\partial^{2}Ric+\partial\Gamma*Ric+\Gamma*\partial Ric
+\Gamma*\Gamma*Ric\right)\\
&= g^{-1}*\partial^{2}Ric+\partial\Gamma*Ric+\Gamma*\partial Ric+\Gamma*\Gamma*Ric,
\end{split}
\end{align}
and the term $g^{-1}*\partial^{2}Ric$ has the form 
$g^{kl}\partial_{k}\partial_{l}Ric_{ij}$ which we can also write as
\begin{align}
\Delta_{g_{0}}Ric_{ij}+\left(g^{kl}-\delta^{kl}\right)\partial_{kl}Ric_{ij}.
\end{align}
The terms $\left(g^{kl}-\delta^{kl}\right)\partial_{kl}Ric_{ij}$,
$\partial\Gamma*Ric$, and $\Gamma*\partial Ric$  in (\ref{eqeuc3}) are 
$O(1)$ as $r\rightarrow 0$ and the terms $\Gamma*\Gamma*Ric$ are 
$O(r^{2})$ as $r\rightarrow 0$ as needed. The other cases follow also by induction.
\end{proof}

\begin{proof}[Proof of Theorems \ref{tr1} and \ref{tr2}]
By Lemma \ref{lemsrt1}, we can find a change of coordinates $\phi$ around the
origin such that $\phi_{*}g$ satisfies (\ref{eqsrt8}) and (\ref{eqsrt9}). 
Recall that the obstruction-flat systems have the form 
\begin{align}
\Delta^{k-1}_{g} B_{ij} =\sum_{j= 2}^{k+1}\sum_{\alpha_{1}+\ldots+\alpha_{j}
=2(k+1)-2j}\nabla^{\alpha_{1}}_{g}Rm*\ldots*\nabla^{\alpha_{j}}_{g}Rm.\label{eqsrt10}
\end{align}
From the expression (\ref{eqobs1}) we can write the Bach tensor as
\begin{align}
B_{ij}=\Delta A_{ij}-\nabla_{i}\nabla^{k}A_{kj}+Rm*Rm,
\end{align}
and using that $g$ has constant scalar curvature together with the Bianchi identity 
we can rewrite  (\ref{eqsrt10}) as
\begin{align}
\Delta^{k}_{g} Ric 
= \sum_{j=2}^{k+1}\sum_{\alpha_{1}+\ldots+\alpha_{j}=2(k+1)-2j}
\nabla_{g}^{\alpha_{1}}Rm*\ldots*\nabla_{g}^{\alpha_{j}}Rm,
\label{eqsrt11}
\end{align}
which is exactly the same form as $\eqref{eqsrt11inorb}$. 
From Lemma \ref{lemeuc},  (\ref{eqsrt11}) becomes
\begin{align}
\Delta_{g_{0}}^{k}Ric= T,
\label{eqsrteuc}
\end{align}
where $T =  O(r^{-2(k-1)})$ as $r\rightarrow 0$.
Note that $T\in L^{p}$ near the origin for $p=\infty$ if $k=1$ and for 
any $1\le p<\frac{n}{2(k-1)}$ if 
$k>1$. On the other hand, since $Ric$ is bounded near the origin, $Ric\in L^{p}$ 
for any such $p$. It follows that $Ric(g)$ is a weak solution of 
$\Delta^{k}_{g_{0}}Ric=T$ on $B_{\rho}(0)\backslash\{0\}$ and it is easy to 
prove that in that case $Ric(g)$ extends to a weak solution of (\ref{eqsrteuc}) 
on $B_{\rho}(0)$. We conclude that $Ric\in W^{2k,p}$. Choose 
\begin{align}\label{pchc}
p=\left\{
\begin{array}{ll}
\infty&~\mbox{if}~k=1\\
(1-\epsilon)\frac{n}{2(k-1)}&~\mbox{with}~0<\epsilon<\frac{1}{2k-1}~\mbox{if}~k>1.
\end{array}
\right.
\end{align}
Observe also that $0<2k-1-\frac{n}{p}\le 1$ so by the 
Sobolev inequality, for any $\alpha$ such that $0<\alpha<2k-1-\frac{n}{p}$ we have
\begin{align}
\|\nabla Ric\|_{C^{\alpha}(B_{\rho}(0))}\le C\|\nabla Ric\|_{W^{2k-1,p}(B_{\rho}(0))},
\label{eqsrt11_1}
\end{align}
with $C=C(n,p,k,\rho)$ and then $Ric\in C^{1,\alpha}(B_{\rho}(0))$. Note that 
from the estimates (\ref{eqsrt8}) and (\ref{eqsrt9}) we have $g\in C^{1,\alpha}$, 
which is sufficient for the existence of harmonic coordinates at 
the origin \cite[Lemma 1.2.]{dk}.
In this harmonic coordinate system the metric $g$ is also $C^{1,\alpha}$ and 
solves (\ref{eqsrt11}). Moreover, by (\ref{eqsrt11_1}) and \cite[Corollary 1.4]{dk}, 
$Ric\in C^{1,\alpha}$ near the origin.

We then have that $g$ is a solution of (\ref{eqsrteuc}) and is also a 
solution of an equation of the form
\begin{align}
\frac{1}{2}g^{ij}\partial^{2}_{ij} g_{kl}+Q_{kl}(\partial g, g)
=-Ric_{kl}(g),\label{eqsrt12}
\end{align}
where $Q(\partial g,g)$ is an expression that is quadratic in $\partial g$, 
polynomial in $g$ and has  $\sqrt{|g|}$ in its denominator. Letting $p$ and 
$\alpha$ be as above, we know that $g$ and 
$Ric(g)$ are $C^{1,\alpha}$ at the origin, in particular they both are in 
$W^{1,p}$.  Using elliptic regularity in (\ref{eqsrt12}) we conclude that 
$g \in W^{3,p}$ and the Sobolev inequality (compare (\ref{eqsrt11_1})) implies that 
$g\in C^{2,\alpha}$. Furthermore, since $Ric\in C^{1,\alpha}$ it also follows 
that $g\in C^{3,\alpha}$ (see \cite[Theorem 4.5]{dk}). 
With this regularity in $g$ we can write (\ref{eqsrt11}) as (\ref{eqsrteuc}) 
in harmonic coordinates, i.e., we can write (\ref{eqsrt11}) as an equation 
of the form $\Delta_{g_{0}}^{k}Ric=T'$ with $T'\in L^{p}$.
 
Next, we claim that $\Delta_{g_{0}}^{k} Ric\in W^{1,p}$. To see this, 
take one derivative of \eqref{eqsrt11}. 
Since $g \in C^{3,\alpha}$, one sees that all of the terms on the 
right hand side are $O( r^{-2(k-1)})$ as $r \rightarrow 0$,
which is in $L^p$ for $p$ as in \eqref{pchc}. 
This shows that $\Delta^k_g Ric \in W^{1,p}$. 
Replacing $\Delta^k_g Ric$ with $\Delta^k_{g_0} Ric$
will introduce terms as in Lemma \ref{lemeuc}, but using the fact that 
$Ric$ is now in $C^{1,\alpha}$, we see that these terms are also in $W^{1,p}$,
and consequently  $\Delta_{g_{0}}^{k} Ric\in W^{1,p}$. Elliptic regularity 
then implies that $Ric\in W^{2k+1,p}$. By the Sobolev inequality, 
$Ric\in C^{2,\alpha}$ and then \eqref{eqsrt12} implies $g\in C^{4,\alpha}$. 
It is clear that we can bootstrap the above argument to prove that $g$ 
is smooth at the origin.      
\end{proof} 

\appendix
\section{Turan's Lemma}\label{turan_lemma}
Given complex numbers $z_{1},\ldots,z_{d}$ and a nonnegative integer $l$ we use $S_{l}$ denote the sum 
\begin{align}
c_1, \ldots,c_d \in \mathbb{C}, \ \ S_{l}=\sum_{j=1}^{d}c_{j}z_{j}^{l}.\label{trigsum}
\end{align}

The following lemma is the version of Turan's Lemma that we will use throughout this section:

\begin{lemma}
\label{cortur1} 
Let $z_{1},\ldots,z_{d}$ where $d>1$ be complex numbers with $|z_{j}|\ge 1$ for $j=1,\ldots,d$ and let $m$ be an integer with $m\ge 1$. Then
\begin{align}
|S_{0}|^{2}\le C\max\{|S_{m+1}|^{2},\ldots,|S_{m+d}|^{2}\},
\end{align}
where the positive constant $C=C(m,d)$ can be estimated as
\begin{align}
C\le A(d)\left(\frac{m+d}{d}\right)^{2(d-1)},
\end{align}
for some constant $A(d)$ depending only on $d$.
\end{lemma}
\begin{proof}The proof is a straightforward modification 
of that in \cite[Section 1.1]{Nazarov}. 
\end{proof}

\subsection{An integral form of Turan's Lemma}
Let $h$ be a positive number and consider the arithmetic progression $\{t_{l}=lh\}_{l}$ where $l$ are nonnegative integers. Consider also complex numbers $\zeta_{1},\ldots,\zeta_{d}$ such that $Re(\zeta_{j})\ge 0$ for $j=1,\ldots,d$. If we let $p(t)$ denote
\begin{align}
p(t)=\sum_{j=1}^{d}c_{j}e^{\zeta_{j}t},
\end{align}
then 
\begin{align}
p(lh)=\sum_{j=1}^{d}c_{j}z^{l}_{j}(h),
\end{align}
where $z_{j}(h)=e^{\zeta_{j}h}$. Since  $|z_{j}(h)|\ge 1$ we have from Lemma \ref{cortur1} an inequality of the form
\begin{align}
|p(0)|^{2}\le C(m,d)\left(\max_{l=m+1,\ldots,m+d}\left\{|p(lh)|^{2}\right\}\right)\le
C(m,d)\left(\sum_{l=m+1}^{m+d}|p(lh)|^{2}\right),\label{turint1}
\end{align}
for any integer $m\ge 1$ where $C(m,d)\le A(d)\left(\frac{m+d}{d}\right)^{2(d-1)}$. 
\begin{lemma}\label{lemturint1} Let $\zeta_{1},\ldots,\zeta_{d}$ be complex numbers satisfying $Re(\zeta_{j})\ge 0$ for $j=1,\ldots,d$, and let $p(t)=\sum_{j=1}^{d}c_{j}e^{\zeta_{j}t}$. Then for any positive numbers $0<a<b$
\begin{align}
|p(0)|^{2}\le A(d)\left(\frac{b}{b-a}\right)^{2(d-1)}\frac{(b+a)}{(b-a)^{2}} \int_{a}^{b}|p(t)|^{2}dt.
\end{align}
\end{lemma}

\begin{proof} Let $c=\frac{a+b}{2}$, set $h_{0}=\frac{b-c}{d}=\frac{b-a}{2d}$ and let $m$ be the integer part of $\frac{c}{h_{0}}$ (i.e. the only integer $m$ such that $m\le \frac{c}{h_{0}}< m+1$). Note that in this case $m\ge 1$. 
For any $h>0$ by \eqref{turint1} we have
\begin{align}
|p(0)|^{2}\le A(d)\left(\frac{m+d}{d}\right)^{2(d-1)}\sum_{l=m+1}^{m+d}\left(|p(lh)|^{2}\right).\label{turint2}
\end{align}
By our choice of $m$ and $h_{0}$ we have 
\begin{align}
(m+d)h_{0}\le c+dh_{0}=b,
\end{align}
so then $a\le lh_{0}\le b$ for $l=m+1,\ldots,m+d$. Let $\eta=\frac{2a}{a+b}$, i.e. $c\eta=a$.  Note that $0<\eta<1$. By taking the average of (\ref{turint2}) respect to $h$ on the interval $[\eta h_{0},h_{0}]$ we have
\begin{align}
|p(0)|^{2}\le A(d)\left(\frac{m+d}{d}\right)^{2(d-1)}\frac{1}{h_{0}(1-\eta)}\int_{\eta h_{0}}^{h_{0}}\left(\sum_{l=m+1}^{m+d}|p(lh)|^{2}\right)dh,
\end{align}
It is easy to see that if
$f:\R\mapsto\R$ is a continuous nonnegative function and and 
$T_{0},T_{1}$ are positive numbers with $T_{0}<T_{1}$, then
\begin{align}
\int_{T_{0}}^{T_{1}}\left(\sum_{l=m+1}^{m+d}f(lh)\right)dh\le\frac{d}{2}\int_{(m+1)T_{0}}^{(m+d)T_{1}}f(h)dh.
\end{align} 
Therefore, we have
\begin{align}
\begin{split}
|p(0)|^{2} 
&\le A(d)\frac{d}{2}\left(\frac{m+d}{d}\right)^{2(d-1)}\frac{1}{h_{0}(1-\eta)}\int_{a}^{b}|p(t)|^{2}dt,
\end{split}
\end{align}
where we have used that $(m+1)\eta h_{0}>c\eta=a$. We also have $(1-\eta)h_{0}=\frac{(b-a)^{2}}{2d(b+a)}$. On the other hand
\begin{align}
\left(\frac{m+d}{d}\right)^{2(d-1)}=\left(\frac{(m+d)h_{0}}{dh_{0}}\right)^{2(d-1)}\le 2^{2(d-1)}\left(\frac{b}{b-a}\right)^{2(d-1)},
\end{align} 
so the result follows.
\end{proof}
\begin{corollary}\label{corturint1}
If $p(t)$ is as before then for any $R>0$ we have
\begin{align}
\|p\|^{2}_{L^{\infty}[0,R]}\le \frac{A(d)}{R}\int_{\frac{3R}{2}}^{2R}|p(t)|^{2}dt,\label{turlinfl2}
\end{align}
and also
\begin{align}
\int_{0}^{R}|p(t)|^{2}dt\le A(d)\int_{3\frac{R}{2}}^{2R}|p(t)|^{2}dt.\label{turl2l2}
\end{align}
\end{corollary}
\begin{proof} Let $t_{0}\in[0,R]$ be such that $\|p\|_{L^{\infty}[0,R]}=|p(t_{0})|$, and consider $p_{t_{0}}(\tau)$ given by
\begin{align}
p_{t_{0}}(\tau)=\sum_{j=1}^{d}c_{j}e^{\zeta_{j}t_{0}}e^{\zeta_{j}\tau}.
\end{align}
Clearly $p_{t_{0}}(t-t_{0})=p(t)$ and $|p_{t_{0}}(0)|=\|p\|_{L^{\infty}[0,R]}$. 
By Lemma \ref{lemturint1} we have
\begin{align}
|p_{t_{0}}(0)|^{2}\le 4 \cdot 2^{d-1} A(d)\left(\frac{2R-t_{0}}{R}\right)^{d-1}\left(\frac{\frac{7}{2}R-2t_{0}}{R^{2}}\right)\int_{\frac{3R}{2}-t_{0}}^{2R-t_{0}}|p_{t_{0}}(\tau)|^{2}d\tau.\label{turintshift}
\end{align}
Using a change of variables in (\ref{turintshift}) we obtain
\begin{align}
\|p\|^{2}_{L^{\infty}[0,R]}\le 4 \cdot4^{d-1}A(d)\left(\frac{7}{2R}\right)\int_{\frac{3}{2}R}^{2R}|p(t)|^{2}dt,
\end{align}
which proves (\ref{turlinfl2}). Once we have shown (\ref{turlinfl2}) we can prove (\ref{turl2l2}) by writing
\begin{align}
\frac{1}{R}\int_{0}^{R}|p(t)|^{2}dt\le\|p\|^{2}_{L^{\infty}[0,R]}\le \frac{A(d)}{R}\int_{\frac{3R}{2}}^{2R}|p(t)|^{2}dt.
\end{align}
\end{proof}
\subsection{Proof of Lemma \ref{posneg}}\label{three_annulus_turan}
By making the change of variables $r = e^t$, it 
it is clear from (\ref{L2norm}) that it suffices to show the following lemma 
\begin{lemma}\label{3annulimult}
Let $p(t)$ be a sum of the form 
\begin{align}
p(t)=\sum_{j=1}^{d}\sum_{s=0}^{n_{j}}c_{j,s}t^{s}e^{\zeta_{j}t},
\end{align}
where $n_{j}\ge 0$ are integers, and $c_{j,s} \in \mathbb{C}$ are fixed.
Let $M=\sum_{j=1}^{d}n_{j}$. Then
\begin{enumerate}
\item If $\lambda = \min\{Re(\zeta_{1}),\ldots,Re(\zeta_{d})\}>0$ then for every positive integer $l$,
\begin{align}
e^{\lambda R}\int_{(l-1)R}^{lR}|p(t)|^{2}dt\le A(M+d)\int_{lR}^{(l+1)R}|p(t)|^{2}dt.\label{multgrowth}
\end{align}
\item If $\lambda = \min\{-Re(\zeta_{1}),\ldots,-Re(\zeta_{d})\}>0$, then for any positive integer $l$,
\begin{align}
\int_{lR}^{(l+1)R}|p(t)|^{2}dt\le C(M+d)e^{-\lambda R}\int_{(l-1)R}^{lR}|p(t)|^{2}dt
\label{multdecay}.
\end{align}
\end{enumerate}
\end{lemma}
\begin{proof}
We first consider functions $p:\R\mapsto\C$ that have the form
\begin{align}
p(t)=\sum_{j=1}^{d}c_{j}e^{\zeta_{j}t},\label{nomult}
\end{align}
with $c_{j}\in\C$, i.e., we do not consider the case of roots $\zeta_{j}$ 
with multiplicities.
Suppose that all numbers $\zeta_{j}$, $j=1,\ldots, d$ have positive real part. 
Let us first prove (\ref{multgrowth}) for $l=1$. 
Consider the function $\tilde{p}(t)$ given by
\begin{align}
\tilde{p}(t)=\sum_{j=1}^{d}c_{j}e^{(\zeta_{j}-\lambda)t}.
\end{align} 
Since $Re(\zeta_{j}-\lambda)\ge 0$ for all $j=1,\ldots, d$, we have from 
Corollary \ref{corturint1}
\begin{align}
\int_{0}^{R}|\tilde{p}(t)|^{2}dt\le A(d)\int_{\frac{3R}{2}}^{2R}|\tilde{p}(t)|^{2}dt\label{growth2},
\end{align}
multiplying both sides of (\ref{growth2}) by $e^{3\lambda R}$ we have
\begin{align}
\begin{split}
e^{\lambda R}\int_{0}^{R}e^{2\lambda t}|\tilde{p}(t)|^{2}dt
&\le e^{3\lambda R}\int_{0}^{R}|\tilde{p}(t)|^{2}dt
\le A(d)e^{3\lambda R}\int_{\frac{3R}{2}}^{2R}|\tilde{p}(t)|^{2}dt\\
&\le A(d)\int_{\frac{3R}{2}}^{2R}e^{2\lambda t}|\tilde{p}(t)|^{2}dt
\le A(d)\int_{R}^{2R}e^{2\lambda t}|\tilde{p}(t)|^{2}dt,
\end{split}
\end{align}
and $e^{2\lambda t}|\tilde{p}(t)|^2=|p(t)|^{2}$. For the case $l>1$, we write any $t\in [(l-1)R,(l+1)R]$ as $t=(l-1)R+\tau$ where $\tau\in [0,2R]$ and then we write $p(t)$ as 
\begin{align}
p(t)=q_{l}(\tau)=\sum_{j=1}^{d}c_{j,l}e^{\zeta_{j}\tau},
\end{align}
where $c_{j,l}=c_{j}e^{(l-1)R}$. Applying the above argument to $q_l(\tau)$ then 
(\ref{multgrowth}) follows after a change of variables. If now all numbers $Re(\lambda_{j})$ are negative for $j=1\ldots,d$, it suffices to prove (\ref{multdecay}) for $l=1$ since as before, the general case $l\ge 1$ follows after a change of variables. Let $\tilde{p}(t)=\sum_{j=1}^{d}c_{j}e^{(\zeta_{j}+\lambda)t}$ and write $t\in [0,2R]$ as $t=2R-\tau$ where $\tau\in [0,2R]$, then $c_{j}e^{(\zeta_{j}+\lambda)t}=c_{j}e^{(\zeta_{j}+\lambda)2R}e^{-(\zeta_{j}+\lambda)\tau}$ and $Re(-(\zeta_{j}+\lambda))\ge 0$, so by (\ref{multgrowth}), if we let $q(\tau)=\sum_{j=1}^{d}c_{j}e^{(\zeta_{j}+\lambda)2R}e^{-(\zeta_{j}+\lambda)\tau}$ we obtain
\begin{align}
\label{a28}
\int_{0}^{R}|q(\tau)|^{2}d\tau\le A(d)\int_{\frac{3R}{2}}^{2R}|q(\tau)|^{2}d\tau.
\end{align}
On the other hand,
\begin{align}
\int_{0}^{R}|q(\tau)|^{2}d\tau=\int_{R}^{2R}|\tilde{p}(t)|^{2}dt, 
\mbox{ and } 
\int_{\frac{3R}{2}}^{2R}|q(\tau)|^{2}d\tau=\int_{0}^{\frac{R}{2}}|\tilde{p}(t)|^{2}dt,
\end{align}
so from \eqref{a28}, we have
\begin{align}
\begin{split}
e^{\lambda R}\int_{R}^{2R}e^{-2\lambda t}|\tilde{p}(t)|^{2}dt 
\le e^{-\lambda R}\int_{R}^{2R}|\tilde{p}(t)|^{2}dt
\le A(d)e^{-\lambda R}\int_{0}^{\frac{R}{2}}|\tilde{p}(t)|^{2}dt\\
\le A(d)\int_{0}^{\frac{R}{2}}e^{-2\lambda t}|\tilde{p}(t)|^{2}dt 
\le A(d)\int_{0}^{R}e^{-2\lambda t}|\tilde{p}(t)|^{2}dt,
\end{split}
\end{align}
and $e^{-2\lambda t}|\tilde{p}(t)|^{2}=|p(t)|^{2}$.

For the general case involving multiplicity, 
we will only prove the statement for the case 
$\min\{Re(\zeta_{1}),\ldots,Re(\zeta_{d})\}>0$ 
since the other case will follow as in the above argument. 
We consider first the case corresponding to $n_{1}=1$ and $n_{2}=\ldots=n_{d}=0$. Let $\epsilon>0$ and let $p_{\epsilon}(t)$ be given by
\begin{align}
p_{\epsilon}(t)=c_{1,0}e^{\zeta_{1}t}+\frac{c_{1,1}}{\epsilon}\left(e^{(\zeta_{1}+\epsilon)t}-e^{\zeta_{1}t}\right)+\sum_{j=2}^{d}c_{j,0}e^{\zeta_{j}t},
\end{align}
then by \eqref{multgrowth} for the case with no multiplicity, we have
\begin{align}
e^{\lambda R}\int_{(l-1)R}^{lR}|p_{\epsilon}(t)|^{2}dt\le A(d+1)\int_{lR}^{(l+1)R}|p_{\epsilon}(t)|^{2}dt,\label{multgrowth1}
\end{align}
where $\lambda$ is defined as before, 
and since $A(d+1)$ does not depend on $\epsilon$ we can take the limit of (\ref{multgrowth1}) as $\epsilon$ tends to zero and obtain (\ref{multgrowth}). 
For the higher multiplicity case, (\ref{multgrowth}) can be proved using induction.
\end{proof}
\bibliography{Obstruction}
\end{document}